\documentclass[a4paper,11pt]{amsart}
\usepackage[left=1in,right=1in,top=1in,bottom=1in]{geometry}
\usepackage[utf8]{inputenc}
\usepackage{amsmath}
\usepackage{amsfonts}
\usepackage{amscd}
\usepackage{amssymb,amsthm}
\usepackage{enumerate}
\usepackage{tikz,tikz-cd}
\usepackage{mathtools}

\theoremstyle{plain}
\newtheorem{thm}{Theorem}[section]
\newtheorem{introthm}{Theorem}

\newtheorem{introconj}[introthm]{Conjecture}
\newtheorem{prop}[thm]{Proposition}
\newtheorem{lem}[thm]{Lemma}

\newtheorem{cor}[thm]{Corollary}

\theoremstyle{definition}
\newtheorem{defn}[thm]{Definition}

\newtheorem{problem}[thm]{Problem}
\theoremstyle{remark}
\newtheorem{remark}[thm]{Remark}

\usepackage{tikz}
\usetikzlibrary{shapes,calc,arrows,intersections,decorations.pathreplacing,decorations.markings,shapes.geometric,calc,positioning,backgrounds}

\newcommand{\bC}{{\mathbb C}}

\newcommand{\bE}{{\mathbb E}}

\newcommand{\bN}{{\mathbb N}}

\newcommand{\bQ}{{\mathbb Q}}

\newcommand{\bZ}{{\mathbb Z}}

\newcommand{\cE}{{\mathcal E}}

\newcommand{\cG}{{\mathcal G}}

\newcommand{\cN}{{\mathcal N}}
\newcommand{\cO}{{\mathcal O}}

\newcommand{\cS}{{\mathcal S}}

\newcommand{\cV}{{\mathcal V}}

\newcommand\Det{\operatorname{det}}
\newcommand\ev{\operatorname{ev}}

\newcommand\Fix{{\operatorname{Fix}}}
\newcommand\Gal{{\operatorname{Gal}}}

\newcommand\id{\operatorname{id}}

\newcommand\spec{\operatorname{spec}}

\DeclareMathOperator{\tr}{tr}

\newcommand{\Qbar}{\overline{\bQ}}

\newcommand{\norm}[1]{\left\|#1\right\|}
\newcommand{\set}[1]{\left\{#1\right\}}
\newcommand{\paren}[1]{\left(#1\right)}

\newcommand{\ang}[1]{\left\langle#1\right\rangle}
\newcommand{\ip}[1]{\left\langle#1\right\rangle}

\newcommand{\uu}[1]{{\underline{\underline{#1}}}}

\definecolor{green}{RGB}{13,177,75} %

\DeclareFontFamily{U}{mathb}{\hyphenchar\font45}
\DeclareFontShape{U}{mathb}{m}{n}{
      <5> <6> <7> <8> <9> <10> gen * mathb
      <10.95> mathb10 <12> <14.4> <17.28> <20.74> <24.88> mathb12
      }{} 
\DeclareSymbolFont{mathb}{U}{mathb}{m}{n}
\DeclareFontSubstitution{U}{mathb}{m}{n}
\DeclareMathSymbol{\Asterisk}      {2}{mathb}{"06}

\newcommand{\gp}{\mathop{\begin{tikzpicture}[baseline]
\node[circle, fill=cyan, draw=black, inner sep=0pt, minimum size=3pt] (mid) at (0, 2.5pt) {};
\foreach \x in {0,60,...,300} {
\node[circle, fill=cyan, draw, inner sep=0pt, minimum size=3pt] (\x) at ($(mid)!6pt!(mid.\x)$) {};
\draw (\x) -- (mid) ;
}
\end{tikzpicture}}}

\newcommand{\sstrc}{\begin{tikzpicture}
    \node [inner sep=0, minimum size=0] at (-6pt,0) {};
    \draw (-5pt,0) -- (0,0) node[draw, shade, circle, ball color=black!60!white, inner sep=0pt, minimum size=4pt] {};
    \node at (0,1pt) {};
\end{tikzpicture}}

\usepackage{hyperref}

\title{Strong $1$-boundedness, $L^{2}$-Betti numbers,  algebraic soficity, and graph products}

\author{Ian Charlesworth}
\address{Cardiff University, School of Mathematics, Abacws, Senghennydd Rd, Cardiff CF24 4AG, United Kingdom}
\email{charlesworthi@cardiff.ac.uk}
\urladdr{https://www.ilcharle.com/}

\author{Rolando de Santiago}
\address{Department of Mathematics, Purdue University, Mathematical Sciences Bldg, 150 N University St, West Lafayette, IN 47907}
\email{desantir@purdue.edu}
\urladdr{https://www.math.purdue.edu/~desantir}

\author{Ben Hayes}
\address{Department of Mathematics,
University of Virginia, 141 Cabell Drive, Kerchof Hall, Charlottesville, VA, 22904}
\email{brh5c@virginia.edu}
\urladdr{https://sites.google.com/site/benhayeshomepage/home}

\author{David Jekel}
\address{\parbox{\linewidth}{Department of Mathematics, University of California, \\
San Diego, 9500 Gilman Drive \# 0112, La Jolla, CA 92093}}
\email{djekel@ucsd.edu}
\urladdr{http://davidjekel.com}

\author{Srivatsav Kunnawalkam Elayavalli}
\address{\parbox{\linewidth}{Department of Mathematics, University of California, \\
San Diego, 9500 Gilman Drive \# 0112, La Jolla, CA 92093}}
\email{srivatsav.kunnawalkam.elayavalli@vanderbilt.edu}
\urladdr{https://sites.google.com/view/srivatsavke}

\author{Brent Nelson}
\address{Department of Mathematics, Michigan State University, 619 Red Cedar Road, C212 Wells Hall, East Lansing, MI 48824}
\email{brent@math.msu.edu}
\urladdr{https://users.math.msu.edu/users/banelson/}
\begin{document}

\maketitle

\begin{abstract}
    We show that graph products of non trivial finite dimensional von Neumann algebras are strongly 1-bounded when the underlying $*$-algebra has vanishing first $L^2$-Betti number. The proof uses a combination of the following two key ideas to obtain lower bounds on the Fuglede--Kadison determinant of  matrix polynomials in a generating set: a  notion called ``algebraic soficity" for $*$-algebras allowing for the existence of Galois bounded microstates with asymptotically constant diagonals; a probabilistic construction of the authors of permutation models for graph independence over the diagonal. 
\end{abstract}

\section{Introduction}

Finite dimensional approximations of infinite dimensional objects are a common theme in analysis, dynamics, and operator algebras.
In the context of groups, they arise in both soficity and Connes embeddability of the group von Neumann algebra (sometimes referred to as hyperlinearity), which are the ability to be approximated by permutations or finite dimensional unitary matrices, respectively.
Connes-embeddable tracial von Neumann algebras are those which admit matrix approximations in a weak sense.
The quantum complexity result announced in \cite{ji2020mipre} implies that not all tracial von Neumann algebras have this property; among Connes-embeddable von Neumann algebras, some---such as free products---have an abundance of matrix approximations which can be constructed probabilistically through random matrix theory, while others---such as amenable von Neumann algebras, property (T) von Neumann algebras, or von Neumann algebras with Cartan subalgebras---have very few matrix approximations.
More precisely, the latter are strongly $1$-bounded in the sense of Jung, or have $1$-bounded entropy $h(M) < \infty$ in the sense of Hayes \cite{Hayes2018}.
Von Neumann algebras with $h(M) = \infty$ enjoy strong indecomposability properties: for instance, they are unable to be decomposed non-trivially as a tensor product, a crossed product, or a join of amenable subalgebras with diffuse intersection; more generally, they cannot be decomposed as a join of subalgebras with finite $1$-bounded entropy.
Using 1-bounded entropy techniques to study the structure of II$_1$ factors (especially, free group factors) has recently been quite fruitful to approach  open problems (see for instance \cite{hayespt, belinschi2022strong, bordenave2023norm, chifan2022nonelementarily}).

Graph products of groups, defined by Green in \cite{Gr90}, are free products of groups indexed by the vertices of a graph, modulo the relations that $\Gamma_v$ and $\Gamma_w$ commute when $v$ and $w$ are adjacent vertices in the graph.
Graph products of von Neumann algebras were introduced by M\l{}otkowski in \cite{Mlot2004} under a different name, then reintroduced and further studied by Caspers and Fima in \cite{CaFi17}.
From a probabilistic viewpoint, graph products give rise to a notion of ``graph independence'', which is a natural way to mix together classical independence and free independence \cite{Mlot2004,SpWy2016}. 

Preservation of Connes-embeddability by graph products was proved by Caspers \cite{casperscep}.
Collins and Charlesworth described how to construct random matrix approximations for a graph product out of given random matrix approximations for the individual algebras $M_v$ \cite{CC2021}.
But despite the matrix approximations being defined by similar techniques as for free products, it was not clear when graph products would have abundant matrix approximations in the sense that $h(M) = \infty$, because the matrix approximations were constructed in a subspace of $M_{N^k}(\bC)$ with much lower dimension than the ambient space.
In this paper we make progress toward classifying when a graph product has $h(M)<\infty$, which can be summarized in the following theorem.
Here items \ref{item: connected graphs intro} and \ref{item: disconnected graphs intro} give a complete characterization of when $h(M) < \infty$ for the case when $M_v$ is diffuse for every $v$, whereas item \ref{item: finite-dimensional case} applies in the much more subtle case when each $M_v$ is finite dimensional.

\begin{introthm}[{Section \ref{subsec: proof of main theorem}}]\label{thm: main theorem intro0}
Let $\cG=(\cV,\cE)$ be a graph with $\#\cV > 1$, and for each $v\in \cV$, let $(M_{v},\tau_{v})$ be a tracial $*$-algebra. Let $(M,\tau)=\gp_{v\in \cG}(M_{v},\tau_{v})$ be their graph product over $\cG$.
\begin{enumerate}[(1)]
    \item Suppose each $M_{v}$ is finite dimensional, and  the trace of every central projection in $M_{v}$ is a rational number. Let $A$ be the $*$-subalgebra of $M$ generated by $\bigcup_{v\in \cV}M_{v}$. If $\beta^{1}_{(2)}(A,\tau)=0$, then $M$ is strongly $1$-bounded.

    \label{item: finite-dimensional case}
    \item If each $M_{v}$ is diffuse and $\cG$ is connected, then $M$ is strongly $1$-bounded (in fact has $1$-bounded entropy at most zero). \label{item: connected graphs intro}
    \item If each $M_{v}$ is diffuse and Connes embeddable, and $\cG$ is disconnected, then $M$ is not strongly $1$-bounded.
   
    \label{item: disconnected graphs intro}
\end{enumerate}
\end{introthm}

We remark that in (\ref{item: disconnected graphs intro}) we show something stronger: there is a (potentially infinite) tuple $x$ of self-adjoint elements of $M$ so that $W^{*}(x)=M$ and $\delta_{0}(x)>1$. As we show in Theorem~\ref{thm: gp of diffuse}, the second two items in the above theorem can be deduced quickly from the known robust properties of $1$-bounded entropy.
We turn our attention, instead, to the question of strongly $1$-boundedness for graph products of finite dimensional algebras.
This focus motivates the results in part (\ref{item: finite-dimensional case}), which is more subtle; developing the tools for its proof occupies the bulk of this paper, and in the end we are able to prove a more general statement in \ref{thm: gp of as is s1b}.

The final step in our proof is to invoke results of Jung \cite[Theorem 6.9]{JungRegularity} and Shlyakhtenko \cite[Theorem 3.2]{Shl2015}, which apply when an operator naturally arising from ``generators and relations'' has positive Fuglede--Kadison pseudo-determinant.
If each $(M_{v},\tau_{v})$ were a group algebra, it would suffice to note that the graph product is sofic by \cite{graphproductofsoficissofic}: the relation matrix in question is a matrix over the rational group ring, and such matrices have positive Fuglede--Kadison pseudo-determinant by \cite{ElekSzaboDeterminant}.
However, extending this to arbitrary $*$-algebras not arising from groups presents substantial challenges.

Our approach is to introduce a notion of \emph{algebraic soficity} for tracial $*$-algebras inspired by soficity of groups. This notion of algebraic soficity ensures that if $y\in M_{n}(A)$ can be expressed as a matrix of polynomials in $x$ with ``nice'' coefficients (e.g. rational, algebraic, etc.), then $y$ has positive Fuglede--Kadison pseudo-determinant.
Crucially, we establish that this notion of algebraic soficity is closed under graph products.


In order to prove positivity of these Fuglede--Kadison determinants we use our notion of algebraic soficity which, while akin to that of soficity of groups, is not a simple translation of the group case.
One na\"ive approach, sufficient to force the appropriate determinants to be positive, would be to require matrix approximations for our generating tuple $x$ with integer entries, chosen so that polynomials in these matrix approximations asymptotically have constant diagonals.
However, this is impossible even for matrix algebras with matrix units as generators: there are very few projections with integral entries, and such projections do not have constant diagonals unless they are scalars.
For similar reasons, it is too much to ask for tracial $*$-algebras which are only slight modifications of group algebras, such as group algebras twisted by an $S^{1}$-valued $2$-cocycle or group measure-space constructions.

We relax this na\"ive approach by only requiring our matrix approximations to have algebraic integer entries.
In order to obtain a lower bound on the pseudo-determinant of these approximations (and thus a lower bound on the Fuglede--Kadison pseudo-determinant of the limiting operator), we use an algebraic number theory argument analogous to \cite[Theorem 4.3]{ThomDiophantine} which considers the Galois conjugates of a matrix with algebraic integer entries and converts upper bounds on the number of such Galois conjugates and of the operator norm of these conjugates into lower bounds on the pseudo-determinant.
We call matrix approximations \emph{Galois bounded microstates} when they have algebraic integer entries, a uniform  bound on the number of their Galois conjugates, and a uniform  bound on the operator norm of these Galois conjugates (see Definition \ref{def: Galois bounded} for the precise definition).
These are the key to our notion of algebraic soficity: a tracial $*$-algebra $(A, \tau)$ is \emph{algebraically sofic} when it has a generating tuple $x$ which admits Galois bounded microstates with asymptotically constant diagonals.
It turns out that all finite dimensional tracial *-algebras are algebraically sofic (see Theorem~\ref{thm: finite dim algebras are alg sofic}).
From the above discussion, we realize that our proof of Theorem~\ref{thm: main theorem intro0}\ref{item: finite-dimensional case} reduces to the following two results.

 \begin{introthm}[{Theorem \ref{item:FKD positive repeat}}]\label{thm:FKD positive intro}
If $(M,\tau)$ is a tracial von Neumann algebra, and $x$ is a generating tuple for $M$ with a Galois bounded sequence of microstates, then for any matrix polynomial in $x$ with algebraic coefficients, the Fuglede--Kadison pseudo-determinant is positive.
\end{introthm}

The proof of Theorem \ref{thm:FKD positive intro} follows by adapting methods of Thom \cite[Theorem 4.3]{ThomDiophantine}.
In order to prove that graph products of finite dimensional algebras have Galois bounded microstates, we prove the following.

\begin{introthm}[{Theorem \ref{thm: gp of as is as}}]\label{thm: algebraic soficity intro}
Algebraic soficity is preserved by graph products.
\end{introthm}

Theorem \ref{thm: algebraic soficity intro} is just a restatement of Theorem~\ref{thm: gp of as is as}, which is proved in Section~\ref{sec: gp of as}.
As a consequence of Theorems \ref{thm: algebraic soficity intro} and \ref{thm:FKD positive intro}, we can replace ``finite dimensional" in Theorem \ref{thm: main theorem intro0} (\ref{item: finite-dimensional case}) with ``algebraically  sofic," under a technical condition on traces (see Theorem \ref{thm: gp of as is s1b} for more details).

Sofic groups themselves have seen numerous applications in recent years: their Bernoulli shift actions can be completely classified (by \cite{Bow, BowenOrn, SewardOrn}); they are known to satisfy the determinant conjecture \cite{ElekSzaboDeterminant}(a conjecture arising in the theory of $L^{2}$-invariants) and consequently their $L^{2}$-torsion is well-defined \cite{LuckBook}; they are known to admit a version of L\"{u}ck approximation \cite{ElekSzaboDeterminant}; they satisfy Gottschalk's surjunctivity conjecture \cite{GromovSurjun}; and they are known to satisfy Kaplansky's direct finiteness conjecture \cite{ElekSzaboDF}.
In fact, any group for which one of these properties is known is also known to be sofic; it is a large open question whether or not every group is sofic. We refer the reader to \cite{LewisICM} for further applications of sofic groups, particularly to ergodic theory.

We expect that our new notion of algebraic soficity will have many similar applications in the theory of von Neumann algebras. Motivated by our work, and using our new notion of algebraic soficity, we make the following conjecture.

\begin{introconj}\label{alg sofic and L2b conj}
Let $(M,\tau)$ be a tracial von Neumann algebra.
Assume that $M$ has a weak$^{*}$-dense, finitely presented, algebraically sofic, unital $*$-subalgebra $A$.
Then $(M,\tau)$ is strongly $1$-bounded if and only if $\beta^{1}_{(2)}(A,\tau)=0$.
\end{introconj}

The fact that if $\beta^{1}_{(2)}(A,\tau)=0$ and $(A,\tau)$ is algebraically sofic, then $(M,\tau)$ is strongly $1$-bounded is a consequence of \cite[Theorem 2.5]{Shl2015} (see e.g. the proof of Theorem \ref{thm: main theorem intro0} (\ref{item: finite-dimensional case}) in Section \ref{subsec: proof of main theorem}).
So the difficulty is in establishing the converse. Partial progress on this has already been made in \cite{DimaLowerEstimates}, and as discussed there an inherent part of the difficulty is in exponentiating a derivation to get a one-parameter family of deformations of $M$ which ``move in a free direction". This conjecture is already interesting to investigate when $(M,\tau)$ is the graph product of finite dimensional tracial von Neumann algebras, and $A$ is the $*$-algebra generated by the vertex algebras.

\begin{remark}
The problem of studying the first $L^2$-Betti numbers for graph products of finite groups has been studied extensively in \cite{DDJO, DO, DO2}. In particular the authors specify that there are algorithms to compute the first $L^2$ Betti numbers for certain graph products of finite groups.  Combining this with \cite{Shl2015} should give examples of strongly $1$-bounded group von Neumann algebras. One would expect that these algorithms would generalize to the setting of finite dimensional $*$-algebras, in which case one could use them in combination with Theorem~\ref{thm: main theorem intro0}(1) to obtain examples of strongly 1-bounded von Neumann algebras not coming from groups.
\end{remark}

One special case of Conjecture \ref{alg sofic and L2b conj} that is worth studying is the case where $M=L(\Gamma)$ is the von Neumann algebra of a group $\Gamma$. In this case, we would expect that if $\Gamma$ is sofic, then $L(\Gamma)$ is strongly $1$-bounded if and only if $\beta^{1}_{(2)}(\Gamma)=0$. This is of particular interest for graph products of groups, because of the aforementioned results that give an algorithmic approach to computing their first $L^{2}$-Betti number. For the special setting of group von Neumann algebras, Conjecture \ref{alg sofic and L2b conj} would follow immediately if the following problem has an affirmative answer.

\begin{problem}\label{prob: free product problem}
Suppose that $\Gamma$ is a group with positive first $L^{2}$-Betti number. Is it true that $L(\Gamma)$ has a finite index subalgebra which decomposes as a free product of two tracial von Neumann algebras $M_{1},M_{2}$?
\end{problem}
Note that if a  tracial von Neumann algebra is a nontrivial free product up to finite index, then it has no Cartan subalgebras  (\cite{Jung2007}, \cite{CartanAFP}). Note also that absence of Cartan for various subfamilies of groups with positive first $L^{2}$-Betti number has been obtained in the literature (see  for instance \cite{ChifanSinclair, CSUProduct, AdrianCompactActions, PopaVaesFree, PopaVaesHyp, SinclairGuassian}), using deformation/rigidity theory.
An affirmative answer to Problem \ref{prob: free product problem} would of course be a surprising structural property of group von Neumann algebras with positive first $L^{2}$-Betti number. However it is not a possibility that should be ruled out. 

\subsection*{Acknowledgements}  We thank Dimitri Shlyakhtenko for  lively discussions; and we thank IPAM and the Lake Arrowhead Conference Center for hosting the Quantitative Linear Algebra long program second reunion conference, where some of these discussions took place.  We thank the American Institute of Mathematics SQuaRES program for hosting us for a week each in April 2022 and April 2023 to collaborate on this project. BH was supported by the NSF grant DMS-2000105. IC was supported by long term structural funding in the form of a Methusalem grant from the Flemish Government. DJ was supported by postdoctoral fellowship from the National Science Foundation (DMS-2002826). BN was supported by NSF grant DMS-1856683.

\tableofcontents

\section{Preliminaries}







\begin{defn}
A \emph{(simple undirected) graph} is a pair $\cG = (\cV, \cE)$ where $\cV$ is a finite set consisting of \emph{vertices}, and $\cE \subset \cV \times \cV$ is a set of \emph{edges}.
We insist that $\cE$ is symmetric (i.e., $(x, y) \in \cE$ if and only if $(y, x) \in \cE$) and non-reflexive (i.e., $(x, x) \notin \cE$ for any $x \in \cV$; that is, we do not allow self-loops).

An \emph{(undirected) multigraph} is the same except $\cE$ is a multiset, allowing parallel edges.
\end{defn}


\subsection{Graph products}

We will now define the graph product of von Neumann algebras, and some important related notions.
Given a graph $\cG = (\cV, \cE)$ and a collection of finite tracial von Neumann algebras $(M_v, \tau_v)$ for each $v \in \cV$,
the graph product will be constructed as finite von Neumann algebra containing a copy of each $M_v$ in such a way that $M_v$ and $M_{v'}$ are in tensor product position if $(v, v') \in \cE$, and in free position otherwise.

\begin{defn} \label{def: G reduced word}
Let $\cG = (\cV, \cE)$ be a graph.
We say that a word $(v_1, \ldots, v_n) \in \cV^n$ is \emph{$\cG$-reduced} provided that whenever $i < k$ are such that $v_i = v_k$, there is some $j$ with $i < j < k$ so that $(v_i, v_j) \notin \cE$.
\end{defn}

If $(v_1, \ldots, v_n) \in \cV^n$ is such a word and $x_i \in M_{v_i}$, then saying the word is \emph{not} $\cG$-reduced is exactly saying that two $x_i$'s from the same algebra could be permuted next to each other and multiplied, giving a shorter word.

\begin{defn}
Let $\cG = (\cV, \cE)$ be a graph, $(M, \tau)$ be a tracial von Neumann algebra, and for each $v \in \cV$ let $M_v \subseteq M$ be a von Neumann sub-algebra.
Then the algebras $(M_v)_{v \in \cV}$ are said to be \emph{$\cG$-independent} if: $M_v$ and $M_w$ commute whenever $(v, w) \in \cE$; and whenever $(v_1, \ldots, v_n)$ is a $\cG$-reduced word and $x_1, \ldots, x_n \in M$ are such that $x_i \in M_{v_i}$ and $\tau(x_i) = 0$, we have $\tau(x_1\cdots x_n) = 0$.

Conversely, given a collection $(M_v, \tau_v)$ of tracial von Neumann algebras, their \emph{graph product} $(M, \tau) = \gp_{v \in \cG} (M_v, \tau_v)$ is the von Neumann algebra generated by copies of each $M_v$ which are $\cG$-independent, so that $\tau|_{M_v} = \tau_v$.
(That the graph product exists and is unique was shown in \cite{Mlot2004}).
When the trace is clear from context, we may write simply $M = \gp_{v \in \cG} M_v$.
\end{defn}
Notice that if $\cG = (\cV, \emptyset)$ then $\gp_{v \in \cG} M_v = \Asterisk_{v \in \cV} M_v$; on the other hand, if $\cG$ is a complete graph, then $\gp_{v \in \cG} M_v = \bigotimes_{v \in \cV} M_v$.

\subsection{Laws in tracial von Neumann algebras}
A tracial von Neumann algebra is a pair $(M,\tau)$ where $M$ is a von Neumann algebra and $\tau\colon M\to \bC$ is a faithful, normal, tracial state. If $a\in M$ is a normal element, we let $\mu_{a}$ be the Borel probability measure supported on the spectrum of $a$ defined by
\[\mu_{a}(E)=\tau(1_{E}(a)) \textnormal{ for all Borel $E\subseteq \bC$}.\]
We then necessarily have that
\[\tau(f(a))=\int f\,d\mu_{a}\]
for all complex-valued, bounded Borel functions $f$ defined on the spectrum of $a$.

Given an integer $r\geq 1$, we let $\bC\ang{T_{1},T_{1}^{*},T_{2},\cdots,T_{r},T_{r}^{*}}$ be the algebra of noncommutative $*$-polynomials in $r$-variables (i.e. the universal $*$-algebra in $r$-variables).
Given a $*$-algebra $A$ and a tuple $a=(a_{1},\cdots,a_{r})\in A^{r}$, and $P\in \bC\ang{T_{1},T_{1}^{*},T_{2},\cdots,T_{r},T_{r}^{*}}$, we use $P(a)$ for the image of $P$ under the unique $*$-homomorphism $\bC\ang{T_{1},T_{1}^{*},T_{2},\cdots,T_{r},T_{r}^{*}}\to A$ which sends $T_{j}$ to $a_{j}$. For later use, if \[P=(P_{ij})_{1\leq i\leq m,1\leq j\leq n}\in M_{m,n}(\bC\ang{T_{1},T_{1}^{*},T_{2},T_{2}^{*},\cdots,T_{r},T_{r}^{*}})\] we define $P(x)\in M_{m,n}(A)$ by
\[(P(x))_{ij}=P_{ij}(x).\]

If $(M,\tau)$ is a tracial von Neumann algebra, and $x\in M^{r}$ is a tuple, we define its \emph{law} $\ell_{x}$ to be the linear functional
\begin{align*}
\ell_{x}\colon\bC\ang{T_{1},T_{1}^{*},T_{2},\cdots,T_{r},T_{r}^{*}} &\to \bC \\
P &\mapsto\tau(P(x)).
\end{align*}
If $n\in \bN$ and $A\in M_{n}(\bC)$ then we let $\ell_{A}$ be the law of $A$ with respect to the normalized  tracial state $\tr_{n}$ on $M_{n}(\bC)$, namely
\[\tr_{n}(A)=\frac{1}{n}\sum_{j=1}^{n}A_{jj}.\]
Suppose we are given a sequence $(M_{n},\tau_{n})$ of tracial von Neumann algebras, and $a_{n}\in M_{n}^{r}$. If $(M,\tau)$ is a tracial von Neumann algebra and $a\in M^{r}$ we say that $\ell_{a_{n}}\to \ell_{a}$ if for all $P\in \bC\ang{T_{1},T_{1}^{*},T_{2},\cdots,T_{r},T_{r}^{*}}$ we have
\[\ell_{a}(P)=\lim_{n\to\infty}\ell_{a_{n}}(P).\]

Laws are spectral measures are related by the following fact: suppose we are given
\begin{itemize}
    \item $(M_{n},\tau_{n})$ are tracial von Neumann algebras,
    \item a $C>0$ and an integer $r\geq 1$
    \item a sequence $a_{n}\in (M_{n})_{s.a.}^{r}$ with $\|a_{n}\|\leq C$.
    \item a tracial von Neumann algebra $(M,\tau)$ and $a\in M_{s.a.}^{r}$
\end{itemize}
Then $\ell_{a_{n}}\to_{n\to\infty}\ell_{a}$ in law if and only if for every self-adjoint $P=P^{*}\in \bC\ang{T_{1},T_{1}^{*},\cdots,T_{r},T_{r}^{*}}$ we have $\mu_{P(a_{n})}\to\mu_{P(a)}$ weak$^{*}$. Moreover, if $r=1$ these conditions are equivalent to saying that $\mu_{a_{n}}\to \mu_{a}$ weak$^{*}$. The proof of this fact is an exercise in applying the Stone-Weierstrass theorem.

If $(M,\tau)$ is a tracial von Neumann algebra, and $x\in M_{m,n}(M)$, we define the \textbf{Fuglede--Kadison pseudo-determinant of $x$} by
\[\Det_{M}^{+}(x)=\exp\left(n\int_{(0,\infty)} \log t\,d\mu_{|x|}(t)\right),\]
where $|x|=(x^{*}x)^{1/2}$, and $\mu_{|x|}$ is the spectral measure with respect to $\tr_{n}\otimes \tau$.  Here we are following the usual convention that $\exp(-\infty)=0$.

\subsection{Galois theory} \label{subsec: Galois theory}
We fix some notation and recalling some of the fundamental concepts of Galois theory, specific to algebraic field extensions of $\bQ$.
Let $\Qbar$ be the algebraic numbers in $\bC$, this is a field by \cite[Corollary 19 in Section 13.2]{DummitUndFoote}
We will write $\cO$ for the algebraic integers in $\bC$; recall that $x\in \cO$ if there is a monic $p\in \bZ[T]$ so that $p(x)=0$.

The \textbf{absolute Galois group of $\bQ$} is defined to be the group $\Gal(\Qbar/\bQ)$ of all field automorphisms of $\Qbar$ (note that such automorphisms automatically fix $\bQ$).
Each $x \in \Qbar$ has finite orbit $\Gal(\Qbar/\bQ)\cdot x := \set{\sigma(x) \colon \sigma \in \Gal(\Qbar/\bQ)}$; if we equip these sets with their discrete topologies, then $\prod_{x \in \Qbar} (\Gal(\Qbar/\bQ)\cdot x)$ is compact by Tychonoff's Theorem, and contains $\Gal(\Qbar/\bQ)$.
Since $\Gal(\Qbar/\bQ)$ is closed in this topology, it is a compact group.
Note that a sequence $\sigma_{n}\in \Gal(\Qbar/\bQ)$ converges to $\sigma\in \Gal(\Qbar/\bQ)$ if for every $x\in \Qbar$ we have $\sigma_{n}(x)=\sigma(x)$ for all sufficiently large $n$.

Though we will not need it, we remark to the reader that the usual Galois correspondence between subgroups and subfields extends to this setting.
Namely, the \emph{closed} normal subgroups are in natural bijection with the Galois extensions of $\bQ$, via the correspondence that sends a closed, normal subgroup $H$ of $\Gal(\Qbar/\bQ)$ to $\Fix_{H}(\Qbar)=\{x\in \Qbar:\sigma(x)=x\textnormal{ for all $\sigma\in H$}\}$.
This correspondence naturally induces an isomorphism $\Gal(\Qbar/\bQ)/H\cong \Gal(\Fix_{H}(F)/\bQ)$.
In particular, if $[\Gal(\Qbar/\bQ):H]<+\infty$, then $\Fix_{H}(\Qbar)$ is a finite Galois extension with degree $[\Gal(\Qbar/\bQ):H]$.

We remind the reader here some of the core results of Galois theory and algebraic number theory, which we will use in Section \ref{sec: GB and alg sofic}.
\begin{enumerate}
    \item If $x\in\Qbar$, then $x\in \bQ$ if and only if $\sigma(x)=x$ for all $\sigma\in \Gal(\Qbar/\bQ)$ (see \cite[Theorem 1.2]{LangBook}),
    \item the algebraic integers form a subring of $\Qbar$ (see \cite[Corollary 24 in Section 15.3]{DummitUndFoote}),
    \item $\cO\cap \bQ=\bZ$  (see \cite[Proposition 28 in Section 15.3]{DummitUndFoote}).
\end{enumerate}

\section{Galois bounded microstates and algebraic soficity}\label{sec: GB and alg sofic}

In this section we introduce the concepts of Galois bounded sequences of microstates and algebraic soficity.
The motivation is to find an analogue of soficity which is better adapted to $*$-algebras not necessarily arising from groups, which will still be sufficient to give us bounds on certain Fuglede--Kadison pseudodeterminants of operators arising from such algebras.

The generators of a sofic group admit microstates among the permutation matrices, where all the entries are $0$ or $1$.
This suffices to prove
that their group algebras satisfy
the determinant conjecture \cite{ElekSzaboDeterminant} (in turn implying that $L^{2}$-torsion of modules over them is well-defined \cite{LuckBook}), as well as L\"{u}ck approximation \cite{ElekSzaboDeterminant}. They are thus of fundamental importance in the study of $L^{2}$-invariance.
We will see that the same sort of control can be obtained when a $*$-algebra has generators admitting microstates whose entries, rather than being integers, are algebraic integers all living in a fixed finite extension of $\bQ$.
We make these idea precise in Definitions~\ref{def: Galois bounded} and \ref{def:alg sofic}.

As motivating examples, we show below that $M_n(\bC)$ is algebraically sofic (despite not being a group von Neumann algebra), as is  $L(\Gamma)$ for any sofic group $\Gamma$.

\subsection{Galois bounded microstates and the Fuglede--Kadison determinant}

Given $A\in M_N(\Qbar)$ and $\sigma\in \Gal(\Qbar/\bQ)$, we write $\sigma(A)$ to mean the matrix obtained by applying $\sigma$ to each entry of $A$. For $\sigma\in \Gal(\Qbar/\bQ)$, we let $\widetilde{\sigma}\in \Gal(\Qbar/\bQ)$ be given by $\widetilde{\sigma}(z)=\overline{\sigma(\overline{z})}$.
Note that if $A\in M_{N}(\Qbar)$, then \begin{equation}\label{eqn: adjoint Galois eqn}
\sigma(A^{*})=\widetilde{\sigma}(A)^{*}.
\end{equation}
This will be used frequently in this section. For a matrix $A\in M_{N}(\Qbar)$, we use
\[\Gal(\Qbar/\bQ)\cdot A=\{\sigma(A):\sigma\in \Gal(\Qbar/\bQ)\}.\]
The following lemma allows us to use number theory to obtain lower estimates on pseudo-determinants of finite dimensional matrices. This lemma will then motivate a special type of microstates approximation sequence whose existence implies positivity of Fuglede--Kadison pseudo-determinants.

\begin{lem}\label{lem:Bring in the Galois bring in the funk}
Suppose $A\in M_{N}(\mathcal{O})$. Set
\[
C=\max_{\sigma\in \Gal(\Qbar/\bQ)}\|\sigma(A)\|,
\]
\[
d=\#(\Gal(\Qbar/\bQ)\cdot A)
\]
Then
\[
\Det^{+}(A)^{1/N}\geq C^{-d^{2}+1}.
\]
\end{lem}

\begin{proof}
Let $\Omega=\Gal(\Qbar/\bQ)\cdot (A^*A)$.
Following \cite[Theorem 4.3]{ThomDiophantine}, set
\[
B=\bigoplus_{S\in \Omega}S.
\]
Then the characteristic polynomial of $B$ is
\[k_{B}=\prod_{S\in \Omega}k_{S},\]
where $k_S$ is the characteristic polynomial of $S$.
Let $r\in \bN\cup\{0\}$ be such that $k_{A^{*}A}(T)=T^{r}p$, where $p\in \mathcal{O}[T]$ has $p(0)\ne 0$. Then, for $S\in\Omega$, we have $k_{S}(T)=T^{r}p_{S}(t)$ with $p_{S}\in \mathcal{O}[T]$, and $p_{S}(0)\ne 0$. Hence
\[k_{B}(T)=T^{r\#\Omega}\prod_{T\in\Omega}p_{S}(T).\]
Set
\[
q=\prod_{S\in\Omega}p_{S}.
\]
Since $k_{B}$ is invariant under $\Gal(\Qbar/\bQ)$ and has algebraic integer coefficients, we know that its coefficients are in $\bQ\cap \mathcal{O}=\bZ$. It follows that $q\in \bZ[T]$ as well. Further $q(0)\ne 0$. Thus $q(0)\in \bZ\setminus\{0\}$ so that
\[1\leq |q(0)|=\Det^{+}(A)^{2}\prod_{S\in\Omega\setminus\{A^{*}A\}}|p_{S}(0)|.\]
For $S\in\Omega$, we know that $p_{S}(0)$ is the product of the nonzero eigenvalues of $S$. For $\sigma\in \Gal(\Qbar/\bQ)$, we define $\widetilde{\sigma}\in \Gal(\Qbar/\bQ)$ by $\widetilde{\sigma}(z)=\overline{\sigma(\overline{z})}$. Then,
\[\|\sigma(A^{*}A)\|=\|\widetilde{\sigma}(A)^{*}\sigma(A)\|\leq C^{2}.\]
This estimate implies that $|p_{S}(0)|\leq C^{2N}$ for every $S\in\Omega$.
So
\[1\leq \Det^{+}(A)^{2}C^{2(\#\Omega-1)N}.\]
Moreover,
(\ref{eqn: adjoint Galois eqn}) implies that
\[\Omega\subseteq\{(\sigma(A))^{*}\phi(A):\sigma,\phi\in \Gal(\Qbar/\bQ)\}\subseteq \{S_{1}^{*}S_{2}:S_{1},S_{2}\in \Gal(\Qbar/\bQ)\cdot A\}.\]
Thus $\#\Omega\leq d^{2}$,
and this completes the proof.
\end{proof}
The preceding lemma suggests the following definition.

\begin{defn} \label{def: Galois bounded}
Let $n(k)$ be a sequence of natural numbers. Let $\Gal(\Qbar/\bQ)$ be the absolute Galois group of $\bQ$. We say that $X^{(k)}\in M_{n(k)}(\bC)$ is \textbf{Galois bounded} if
\begin{itemize}
    \item the entries of $X^{(k)}$ are algebraic integers;
    \item $\sup_{k}\max_{\sigma\in \Gal(\Qbar/\bQ)}\|\sigma(X^{(k)})\|<+\infty$; and
    \item $\#(\Gal(\Qbar/\bQ)\cdot X^{(k)})<+\infty$.
\end{itemize}
If $X^{(k)}\in M_{n(k)}(\bC)^{r}$ we say that it is \textbf{Galois bounded} if $(X^{(k)}_{j})_{k=1}^{\infty}$ is Galois bounded for all $j=1,\cdots,r$.
 If $(M,\tau)$ is a tracial von Neumann algebra, and if $x\in M^{r}$ has $\ell_{X^{(k)}}\to \ell_{x}$, then we say that $X^{(k)}$ are a \textbf{Galois bounded sequence of microstates} for $x$.
\end{defn}

Recalling the correspondence between finite Galois extensions and finite index normal subgroups of the absolute Galois group discussed in \S \ref{subsec: Galois theory}, one can rephrase being Galois bounded in the following way. A sequence $(X^{(k)})_{k=1}^{\infty}\in \prod_{k}M_{n(k)}(\Qbar)$ is Galois bounded if and only if there is a sequence $F_{k}$ of subfields of $\bC$ which are finite Galois extensions of $\bQ$ such that:
\begin{itemize}
    \item we have $X^{(k)}\in M_{n(k)}(F_{k}\cap \mathcal{O})$,
    \item $\sup_{k}[F_{k}:\bQ]<+\infty,$
    \item $\sup_{k}\max_{\sigma\in \Gal(F_{k}/\bQ)}\|\sigma(X^{(k)})\|<+\infty.$
\end{itemize}
In fact, it is possible to rephrase all of our proofs in this framework without any reference to the absolute Galois group. However, phrasing everything in terms of the absolute Galois group makes the setup cleaner and simplifies the proofs of closure of Galois bounded elements under various operations such as multiplication, adjoints, and conjugation by permutations.

\begin{prop}\label{prop:omnibus GB}
Let $S$ be the set of Galois bounded sequences in $\prod_{k}M_{n(k)}(\bC)$.
\begin{enumerate}[(1)]
    \item $S$ is a subring of $\prod_{k}M_{n(k)}(\mathcal{O})$ \label{item:Galois bounded is a ring} which is closed under adjoints, and contains all sequences of the form $(\alpha 1_{n(k)})_{k=1}^{\infty}$ for $\alpha\in \mathcal{O}$. \label{item: Galois bounded ring}

    \item $S$ is invariant under the conjugation action of $\prod_{k}S_{n(k)}$ on $\prod_{k}M_{n(k)}(\Qbar)$. \label{item:Galois bounded perm invariance}
    \item Suppose $X=(X_{k})_{k}\in S$ and $m(k)$ any sequence of integers. If $(Y_{k})_{k}\in \prod_{k}M_{m(k)}(\Qbar)$ is Galois bounded, we have that $(X_{k}\otimes Y_{k})_{k}$ is Galois bounded.
    \label{item:tensors of GB}
    \item If $(X^{(k)})_{k}\in S^{r}$, then for all $P\in \Qbar\ang{T_{1},T_{1}^{*},\cdots,T_{r},T_{r}^{*}}$ we have
    \[\liminf_{k\to\infty}\Det^{+}(P(X^{(k)}))^{1/n(k)}>0.\]
    \label{item: FKD from algebraic integers}
\end{enumerate}
\end{prop}

\begin{proof}
(\ref{item:Galois bounded is a ring}): That the norm boundedness condition is closed under sums and products follows from the facts $\cO$ is a ring, that each $\sigma$ induces an automorphism of $M_{n(k)}(\mathcal{O})$, and that the operator norm is submultiplicative. For the last condition, note that for $A,B\in S$, we have that
\[\Gal(\Qbar/\bQ)\cdot(A^{(k)}B^{(k)})\subseteq \{\sigma(A^{(k)})\phi(B^{(k)}):\sigma,\phi\in \Gal(\Qbar/\bQ)\},\]
with a similar result for the sum. That $S$ contains all constant algebraic integer multiples of the identity is an exercise. Finally, to see that $S$ is closed under adjoints, let $A=(A^{(k)})_{k=1}^{\infty}\in S$.
Equation (\ref{eqn: adjoint Galois eqn})
implies that $((A^{(k)})^{*})_{k=1}^{\infty}\in S$.
The desired result follows.

(\ref{item:Galois bounded perm invariance}): This follows from the fact the action of $\Gal(\Qbar/\bQ)$ on $M_{n(k)}(\mathcal{O})$ commutes with the conjugation action of $S_{n(k)}$.

(\ref{item: FKD from algebraic integers}):
By scaling, we may assume that $P\in \mathcal{O}\ang{T_{1},T_{1}^{*}\cdots,T_{r},T_{r}^{*}}$. By (\ref{item:Galois bounded is a ring}), we know that $(P(X^{(k)}))_{k=1}^{\infty}\in S$.
Set
\[C=\sup_{k}\max_{\sigma\in \Gal(\Qbar/\bQ)}\|\sigma( P(X^{(k)}))\|<+\infty,\]
\[d=\sup_{k}\#\Gal(\Qbar/\bQ)\cdot P(X^{(k)})<+\infty.\]
For each $k$, we then have by Lemma \ref{lem:Bring in the Galois bring in the funk},
\[\Det^{+}(P(X^{(k)}))^{1/n(k)}\geq C^{-d^{2}+1}.\]
Taking limit infimums of both sides completes the proof.

(\ref{item:tensors of GB}): Using (\ref{item:Galois bounded is a ring}) we may reduce to the case that $Y_{k}$ is the $m(k)\times m(k)$ identity matrix. This case is an exercise using, for example,  that
\[\sigma\cdot(A\otimes 1_{m})=(\sigma(A)\otimes 1_{m})\]
for $m,n\in \bN$ and $A\in M_{n}(\Qbar)$.

\end{proof}

We now obtain Theorem \ref{thm:FKD positive intro} from the introduction as a corollary of Proposition \ref{prop:omnibus GB}.

\begin{thm}[{Theorem~\ref{thm:FKD positive intro}}]\label{item:FKD positive repeat}
If $(M,\tau)$ is a tracial von Neumann algebra, and $x$ is a generating tuple for $M$ with a Galois bounded sequence of microstates, then for any matrix polynomial in $x$ with algebraic coefficients, the Fuglede--Kadison pseudo-determinant is positive.
\end{thm}

\begin{proof}
Suppose that $x=(x_{1},\cdots,x_{r})$, and let $P\in M_{m,n}(\Qbar\ang{T_{1},\cdots,T_{r}}).$
Since \[P^{*}P\in M_{n}(\Qbar\ang{T_{1},T_{1}^{*},\cdots,T_{r},T_{r}^{*}})\] we may, and will, assume that case $m=n$.
Let $E_{ij}$ be the standard matrix units in $M_{n}(\bC)$.
It follows by Proposition \ref{prop:omnibus GB} (\ref{item:tensors of GB}) that  the new tuple
\[\widetilde{x}=((x_{l}\otimes E_{ij})_{1\leq i,j\leq n,1\leq l\leq r},(1\otimes E_{ij})_{1\leq i,j\leq n})\]
has a Galois bounded sequence of microstates. Let
\[I=[r]\times [n]\times [n]\sqcup [n]\times [n].\]
For $1\leq i,j\leq n$ we use $(\varnothing,i,j)$ for the element of $I$ which correspond to $(i,j)$ in the copy of  $[n]\times [n]$ inside $I$.
Suppose $P=\sum_{i,j}P_{ij}\otimes E_{ij}$, then as
\[P(x)=\sum_{ij}P_{ij}((x_{l}\otimes E_{ii})_{l=1}^{r})(1\otimes E_{ij}),
\]
we have
\[P(x)=\sum_{i,j}Q_{ij}(\widetilde{x}),\]
where $Q_{ij}\in \Qbar\ang{T_{\beta},T_{\beta}^{*}:\beta\in I}$ is given by
\[Q_{ij}=P_{ij}((T_{(l,i,i)})_{l=1}^{r})T_{\varnothing,i,j}.\]
This construction allows us to reduce to $n=1$, by replacing $x$ with $\widetilde{x}$. Hence we may, and will, assume that $n=1$.

Let $(X^{(k)})_{k}$ be a Galois bounded sequence of microstates for $x$.
The fact that $(X^{(k)})_{k}$ are microstates for $x$ implies that $\mu_{|P(X^{(k)})|}\to \mu_{|P(x)|}$ weak$^{*}$.
Thus, by weak$^{*}$-semicontinuity of integrating the logarithm and (\ref{item: FKD from algebraic integers}),
\begin{align*}
\log \Det^{+}_{M}(P(x))=\int_{(0,\infty)} \log(t)\,d\mu_{|P(x)|}(t)&\geq \liminf_{k\to\infty}\int_{(0,\infty)} \log(t)\,d\mu_{|P(X^{(k)})|}(t)\\
&=\liminf_{k\to\infty}\log \Det^{+}(P(X^{(k)}))^{1/n(k)}>-\infty.\qedhere
\end{align*}

\end{proof}

For later use, we record the fact that the existence of Galois boundedness passes to direct sums.

\begin{lem}\label{lem:direct sums of Galois bounded sequences}
Let $(M_{j},\tau_{j}),j=1,2$ be tracial von Neumann algebras.
Suppose that $x_{j}\in M_{j}^{r_{j}}$ for some $r_{1},r_{2}$ and each $j=1,2$.
Suppose that $(n_{j}(k))_{k=1}^{\infty}$ are sequences of natural numbers for $j=1,2$. Assume we are given for $j=1,2$ microstates sequences $X^{(k)}_{j}\in M_{n_{j}(k)}(\bC)^{r_{j}}$ for $x_{j}$.
Finally, assume that $(t_{k,j})_{k=1}^{\infty},j=1,2$ are sequence of integers so that
\[\alpha=\lim_{k\to\infty}\frac{t_{k,1}n_{1}(k)}{t_{k,1}n_{1}(k)+t_{k,2}n_{2}(k)}\]
exists. Then $((X^{(k)}_{1})^{\oplus t_{k,1}}\oplus 0, 0\oplus (X^{(k)}_{2})^{\oplus t_{k,2}})$ converges in law to the law of $(x_{1}\oplus 0,0\oplus x_{2})$ regarded as an element in $\alpha(M_{1},\tau_{1})\oplus (1-\alpha)(M_{2},\tau_{2})$.
In particular, if there are Galois bounded sequences of microstates for $x_{1},x_{2}$ then for every $0\leq \alpha\leq 1$, there are Galois bounded sequence of microstates for $(x_{1}\oplus 0,0\oplus x_{2})$ regarded as an element in $\alpha(M_{1},\tau_{1})\oplus (1-\alpha)(M_{2},\tau_{2})$.
\end{lem}

\subsection{Algebraic soficity}

\begin{defn}
A sequence of matrices $X^{(k)} \in M_{n(k)}(\bC)$ is called \textbf{asymptotically constant on the diagonal} if
\[
\lim_{k \to \infty} \norm{\Delta_{n(k)}[X^{(k)}] - \tr_{n(k)}[X^{(k)}] 1}_2 = 0,
\]
where $\Delta_{n(k)}$ is the conditional expectation onto the diagonal subalgebra of $M_{n(k)}(\bC)$.
\end{defn}

\begin{defn}
\label{def:alg sofic}
Given a tracial von Neumann algebra $(M,\tau)$ we say that a tuple $x = (x_i)_{i \in I}$ in $M^{I}$ is \textbf{algebraically sofic} if for any finite $F\subseteq I$, there is a sequence of microstates $(X_i^{(k)})_{i \in F}$ for $x\big|_{F}$ that is Galois bounded (Definition \ref{def: Galois bounded}), such that $P(X^{(k)})$ is asymptotically constant on the diagonal for every $*$-polynomial $P$. 
We say that $M$ is \textbf{algebraically sofic} if it has an algebraically sofic generating tuple.

\end{defn}
If $(M,\tau)$ is a tracial von Neumann algebra, and $x\in M^{I}$ is algebraically sofic, we remark that for any set $J$ and any $P\in \mathcal{O}\ang{T_{i},T_{i}^{*}:i\in I}^{J}$ we have that $P(x)$ is algebraically sofic.
The name derives from the case of soficity of groups, as defined by Gromov \cite{GromovSurjun} and named by Weiss \cite{WeissSofic}.
Soficity can be phrased in terms of microstates: a group $\Gamma$ is sofic if and only for every finite $F\subseteq \Gamma$ there  is a sequence $\sigma_{k}\in S_{n(k)}^{F}$ which, when viewed as matrices, form microstates for $F$.
If we equip $S_{n(k)}$ with the metric
\[d(\sigma,\pi)=\frac{1}{n(k)}|\{j:\sigma(j)\ne \pi(j)\}|,\]
then if $\Sigma,\Pi$ are the matrices corresponding to $\sigma,\pi$ we have
\[d(\sigma,\pi)=\frac{1}{2}\|\Sigma-\Pi\|_{2}^{2}.\]
If $F\subseteq \Gamma$ is finite, with $e\in F$ and if $\sigma_{k}\in S_{n(k)}^{F}$ is a microstates sequence for $F$, then for every $g\in F$ we have
\[\tr(\sigma_{k,g})\to_{k\to\infty}\delta_{g=e}\]
Since
\[
\|\Delta_{n(k)}(\sigma_{k,g})-\tr(\sigma_{k,g})\|_{2}^{2}=\|\Delta_{n(k)}(\sigma_{k,g})\|_{2}^{2}-\tr(\sigma_{k,g})^{2}
=\tr(\sigma_{k,g})(1-\tr(\sigma_{k,g}))\] 
being a sequence of microstates forces $\sigma_{k}$ to be asymptotically constant on the diagonal. Thus soficity of $\Gamma$ implies that every tuple in $\Gamma$ is algebraically sofic, when we view $\Gamma\leq \mathcal{U}(L(\Gamma))$.
We record this observation in the following proposition.
\begin{prop}
If $\Gamma$ is a sofic group, then $\set{u_g \colon g \in \Gamma} \subseteq L(\Gamma)$ is algebraically sofic.
In particular, $L(\Gamma)$ is algebraically sofic.
\end{prop}

In the definition of algebraic soficity, we retain having asymptotically constant diagonals, but we relax the requirement of being a permutation (ill-adapted to a nongroup setting).
We instead require entries which are algebraic integers and whose entries have a ``size of integrality" (both in absolute value and in terms of how large of a field extension they live) that is controlled.
The intuition behind this relaxation is that the fact that permutation matrices have integer entries, and the integrality of permutations is used in the proofs of many applications of soficity.

We want to show that finite-dimensional tracial $*$-algebras are algebraically sofic, and to this end, we first show that $M_n(\bC)$ is algebraically sofic using the following group-measure-space construction.
\begin{prop}\label{prop:construction of GBM for matrices}
Let $\Gamma$ be a finite abelian group. Let $(u_{\chi})_{\chi\in\widehat{\Gamma}}$ be the canonical unitaries in $L(\widehat{\Gamma})$.  Consider the action $\alpha$ of $\Gamma$ on $L(\widehat{\Gamma})$ by $\alpha_{g}(u_{\chi})=\chi(g)^{-1}u_{\chi}$ for all $g\in \Gamma,\chi\in \widehat{\Gamma}$.
\begin{enumerate}[(1)]
\item We have $L(\widehat{\Gamma})\rtimes \Gamma\cong M_{|\Gamma|}(\bC)$. \label{item:Takdual}
\item Endow $L(\widehat{\Gamma})\rtimes \Gamma$ with its unique tracial state $\tau$. For $g\in \Gamma$, let $v_{g}$ be the canonical unitaries in $L(\widehat{\Gamma})\rtimes\Gamma$ implementing the action of $\Gamma$. Let
\[\pi\colon L(\widehat{\Gamma})\rtimes \Gamma\to B(L^{2}(L(\widehat{\Gamma})\rtimes \Gamma))\]
be the GNS representation coming from $\tau$. \label{item:GNS matrix entries}
Then:
\begin{enumerate}[(a)]
    \item $\{u_{\chi}v_{g}\colon \chi\in \widehat{\Gamma},\ g\in G\}$ is an orthonormal basis of $L^{2}(L(\widehat{\Gamma})\rtimes \Gamma)$;
    \item if $D$ is the MASA in $B(L^{2}(L(\widehat{\Gamma})\rtimes \Gamma))$ generated by the rank one projections onto $\bC u_\chi v_g$, for $\chi\in \widehat{\Gamma}$ and $g\in \Gamma$, then
    \[\bE_{D}\circ \pi=\tau;\]
    \item the matrix entries of $\pi(u_{\chi}v_{g})$ with respect to $(u_{\theta}v_{h})_{\theta\in \Gamma,h\in G}$ are elements of $\{0\}\cup \{\phi(k):\phi\in \widehat{\Gamma},k\in \Gamma\}$. \label{item:Matri entries algebraic integers 1}
    \item For $1\leq i,j\leq |\Gamma|$, let $E_{ij}$ be the standard matrix units of $M_{|\Gamma|}(\bC)$.  Let $K=\{\phi(k):\phi\in \widehat{\Gamma},k\in \Gamma\}$. Then the isomorphism given in $(\ref{item:Takdual})$ can be chosen so that the matrix entires of  $\pi(E_{ij})$ with respect to $(u_{\theta}v_{h})_{\theta\in \Gamma,h\in \Gamma}$ lie in $\frac{1}{|\Gamma|}\bZ[K]$, for $1\leq i,j\leq \Gamma$.
    \label{item:matrix entries algebraic integers 2}
\end{enumerate}
\end{enumerate}
\end{prop}

\begin{proof}
(\ref{item:Takdual}): The Fourier transform induces an isomorphism $L(\widehat{\Gamma})\cong \ell^{\infty}(\Gamma)$ which conjugates the action of $\Gamma$ on $L(\widehat{\Gamma})$ to the shift action of $\Gamma$ on $\ell^{\infty}(\Gamma)$. This induces an isomorphism $\ell^{\infty}(\Gamma)\rtimes \Gamma\cong L(\widehat{\Gamma})\rtimes \Gamma$, where the action of $\Gamma$ on $\ell^{\infty}(\Gamma)$ is the shift action. The algebra $\ell^{\infty}(\Gamma)\rtimes \Gamma$ is generated by the family of matrix units $\{\delta_g u_{gh^{-1}} \delta_h\colon g,h\in \Gamma\}$ and is therefore isomorphic to $M_{|\Gamma|}(\bC)$.

(\ref{item:GNS matrix entries}): The fact that $\{u_{\chi}v_{g}\colon \chi\in \hat{\Gamma},\ g\in \Gamma\}$ are pairwise orthogonal is a direct computation. We leave it as an exercise to verify that
    \[
        \ang{\pi(u_{\chi}v_{g})u_{\theta}v_{h},u_{\phi}v_{k}} = \theta(g) \delta_{\chi \theta=\phi} \delta_{gh=k},
    \]
for all $\chi\in\widehat{\Gamma},g\in \Gamma$. This implies that
\[\bE_{D}(\pi(u_{\chi}v_{g}))=\tau(u_{\chi}v_{g})1\]
for all $\chi\in\widehat{\Gamma},g\in\Gamma$. Since such elements span $L(\widehat{\Gamma})\rtimes \Gamma$, it follows that $\bE_{D}\circ \pi=\tau$. Part (\ref{item:matri entries algebraic integers 1}) follows from the above computation For part (\ref{item:matrix entries algebraic integers 2}), note that the fact that $(u_{\theta}v_{h})_{\theta,h}$ is an orthonormal basis implies that
\[E_{ij}=\sum_{\chi,g}\tr(E_{ij}(u_{\chi}v_{g})^{*})u_{\chi}v_{g}.\]
As shown above, the matrix entries of $\pi(u_{\chi}v_{g})$ with respect to the basis $(u_{\theta}v_{h})_{\theta,h}$ are in $\bZ[K]$, so the above expansion completes the proof.
\end{proof}

We start by recording how algebraic soficity behaves under tensor products.

\begin{prop}\label{prop: alg sofic tensor}
For $j=1,2$ let $(M_{j},\tau_{j})$ be tracial von Neumann algebras and $x_{j}\in M_{j}^{r_{j}}$ algebraically sofic tuples. Then $(x_{1}\otimes 1,1\otimes x_{2})$ is algebraically sofic.
\end{prop}

\begin{proof}
Let $X_{j}^{(k)}\in M_{n_{j}(k)}(\cO)$ be Galois bounded microstates for $x_{j}$ so that polynomials in $X_{j}^{(k)}$ are asymptotically constant on the diagonal.  By Proposition \ref{prop:omnibus GB} (\ref{item:tensors of GB}) we know that $(X_{1}^{(k)}\otimes 1_{n_{2}(k)},1_{n_{1}(k)}\otimes X_{2}^{(k)})$ is Galois bounded. Monomials in $(X_{1}^{(k)}\otimes 1_{n_{2}(k)},1_{n_{1}(k)}\otimes X_{2}^{(k)})$ are asymptotically constant on the diagonal, and thus polynomials in $(X_{1}^{(k)}\otimes 1_{n_{2}(k)},1_{n_{1}(k)}\otimes X_{2}^{(k)})$ are asymptotically constant on the diagonal.
\end{proof}

This result on tensor products can also be used to show that algebraic soficity is preserved under finite direct sums with rational weights. We show in the next section that the direct sum of two algebraically sofic algebras without rational weights can fail to be algebraically sofic (see Corollary~\ref{cor:non_alg_sof_direct_sum}). 

\begin{thm} \label{thm: finite dim algebras are alg sofic}
Suppose that $(M_{j},\tau_{j})$, $j=1,2$ are algebraically sofic, and let $q\in (0,1)\cap \bQ$. Let $M=M_{1}\oplus M_{2}$ equipped with the trace 
\[\tau(a_{1},a_{2})=q\tau_{1}(a_{1})+(1-q)\tau_{2}(a_{2}).\]
Then $(M,\tau)$ is algebraically sofic.
In particular, finite dimensional tracial von Neumann algebras where every central projection has trace in $\bQ$ are algebraically sofic. 
\end{thm}

\begin{proof}
The ``in particular" part follows from the fact that every finite-dimensional von Neumann algebra is a direct sum of matrix algebras.

Note that $\{\phi(k):k\in \bZ/n\bZ,\phi\in (\bZ/n\bZ)^{\widehat{}}\}=\bZ[e^{2\pi i/n}]$, and $e^{2\pi i/n}$ is algebraic integer.
So the fact that $M_{n}(\bC)$ is algebraically sofic follows from Proposition \ref{prop:construction of GBM for matrices} applied to $\Gamma=\bZ/n\bZ$.  For later use, we note the following specific consequence. For $1\leq i,j\leq n$, let $E_{ij}$ be the standard matrix units of $M_{n}(\bC)$. Then Proposition \ref{prop:construction of GBM for matrices} shows that $(nE_{i,j})_{i,j}$ is an algebraically sofic tuple.

Now let $(A_{j},\tau_{j})_{j=1,2}$ be tracial $*$-algebras. Let $s_{j}\in A_{j}^{r_{j}}$ be a generating tuple for $A_{j},j=1,2$ such that there exists Galois bounded microstates $(X^{(k)}_{j})_{k=1}^{\infty}\in \prod_{k}M_{n_{j}(k)}(\bC)$ as in the definition of algebraic soficity.  

Let $A=A_{1}\oplus A_{2}$ be endowed with the trace
\[\tau(a_{1},a_{2})=t\tau_{1}(a_{1})+(1-t)\tau_{2}(a_{2})\]
for some $t\in \bQ\cap (0,1)$. Write $t=\frac{k}{n}$ with $n\in \bN$ and $0<k<n$. We use the embedding 
\[\pi\colon A_{1}\oplus A_{2}\to M_{n}(\bC)\otimes A_{1}\otimes A_{2}\]
given by $\pi(a_{1},a_{2})=\left(\sum_{i=1}^{k}E_{ii}\right)\otimes a_{1}\otimes 1+\left(\sum_{i=k+1}^{n}E_{ii}\right)\otimes 1\otimes a_{2}.$
It thus suffices to note that Propositions \ref{prop: alg sofic tensor} and \ref{prop:omnibus GB} (\ref{item:Galois bounded is a ring})  implies that
\[n\left(\left(\sum_{i=1}^{k}E_{ii}\right)\otimes s_{1}\otimes 1,n\left(\sum_{i=k+1}^{n}E_{ii}\right)\otimes 1\otimes s_{2}\right)\]
is an algebraically sofic tuple.
\end{proof}

\subsection{Tracial $*$-algebras which are not algebraically sofic}

In this section, we show that certain $*$-algebras can fail to be algebraically sofic. In fact, we show that any self-adjoint element which is algebraically sofic (regarded as a $1$-tuple) must have transcendental trace. Using, we can show that if we equip $A=M_{k_{1}}(\bC)\oplus M_{k_{2}}(\bC)$ with a trace which has a central projection with transcendental trace, then every algebraic sofic element of $A$ must be a scalar multiple of the identity. 

Our starting point is the following result of Thom.

\begin{lem}[Lemma 3.1 of \cite{ThomDiophantine}] \label{lem:the real Thom}
Fix $k\in \bN$ and $C\in [0,+\infty)$. Let $T_{k,c}$ be the set of polynomials in $\bZ[t]$ of degree at most $k$ and whose roots in $\bC$ all have modulus at most $C$. Then $T_{k,C}$ is finite.

\end{lem}

For our purposes, it will be best to rephrase this as follows.
\begin{lem}\label{lem:repharsing THom}
Fix $k\in \bN$ and $C\in [0,+\infty)$. Let $S_{k,C}$ be the set of algebraic integers in $\bC$ which have at most $k$ Galois conjugates, all of which have modulus at most $C$. Then $S_{k,c}$ is finite.
\end{lem}

\begin{proof}
Let $T_{k,c}$ be as in Lemma \ref{lem:the real Thom}. Then
\[S_{k,c}=\bigcup_{p\in T_{k,c}}p^{-1}(\{0\}),\]
so $S_{k,C}$ is a finite union of finite sets.

\end{proof}

\begin{cor}
Let $(M,\tau)$ be a tracial von Neumann algebra and suppose that $x\in M_{s.a.}$ is algebraically sofic. Then $\tau(x)$ is an algebraic integer.
\end{cor}

\begin{proof}
Let $S_{k,c}$ be as in Lemma \ref{lem:repharsing THom}.
Let $X^{(N)}\in M_{k(N)}(\cO)$ be a Galois bounded sequence of microstates which witness that $x$ is algebraically sofic. Since $X^{(N)}$ has asymptotically constant diagonal entries and the average of these entries converges to $\tau(x)$ we may choose $j(N)\in \{1,\cdots,K(N)\}$ with
\[X^{(N)}_{j(N),j(N)}\to_{N\to\infty}\tau(x).\]
By definition of Galois boundedness, there is a $C\in [0,+\infty)$ and a $k\in \bN$ with $X^{(N)}_{j(N),j(N)}\in S_{k,C}$ for all $N$. By Lemma \ref{lem:repharsing THom}, we have that $\tau(x)\in S_{k,C}$ and so $\tau(x)$ is an algebraic integer.

\end{proof}


\begin{thm}\label{thm: store brand Andreas Thom}
 Let $(M,\tau)$ be a tracial von Neumann algebra and $x\in M$ algebraically sofic and self-adjoint. Then all the eigenvalues of $x$ are algebraic integers.   
\end{thm}

\begin{proof}
Let $X^{(N)}\in M_{k(N}(\cO)$ be a sequence of Galois bounded microstates which witness algebraic soficity. By passing to a subsequence, we may assume that there is an $r\in \bN$ with
\[r=|\{\sigma(X^{(N)}):\sigma\in \Gal(\Qbar/\bQ)\}.\]
For each $N\in \bN$, choose $\sigma_{0,N},\cdots,\sigma_{r-1,N}\in \Gal(\Qbar/\bQ)$ such that $\sigma_{0,N}=\id$ and 
\[\{\sigma(X^{(N)}):\sigma\in \Gal(\Qbar/\bQ)\}=\{\sigma_{j,N}(X^{(N)}):j\in \{0,\cdots,r-1\}\}.\]
Set 
\[Y^{(N)}=\bigoplus_{j=0}^{r-1}\sigma(Y^{(j)}).\]
Passing to a further subsequence we may assume that $\mu_{Y^{(N)}}$ weak$^{*}$-converges to a probability measure $\mu$. Let $\mu_{X^{(N)}},\mu_{Y^{(N)}}$ be the spectral measures of $X^{(N)},Y^{(N)}$.  Since the characteristic polynomial of $Y^{(N)}$ is invariant under the absolute Galois group, we know it is an integer and thus $\mu_{Y^{(N)}}$ is an atomic measure supported on algebraic integers, and by Galois boundedness it is supported in a uniformly bounded set. Thus $\mu$ is an integer measure in the sense of \cite{ThomInteger}.
Let $\mu_{x}$ be the spectral measure of $x$. Since $\mu_{X^{(N)}}\leq r\mu_{Y^{(N)}}$ for every $N$ we have that $\mu_{x}\leq r\mu$. If $\lambda\in \bC$ is not an algebraic integer, then since $\mu$ is an integer measure it follows from \cite[Theorem 2.8]{ThomInteger} that 
\[\mu_{x}(\{\lambda\})\leq r\mu(\{\lambda\})=0.\]

\end{proof}


\begin{cor}\label{cor:algebriac weights at algebraics}
Let $(M,\tau)$ be a tracial von Neumann algebra and $x\in M$ algebraically sofic and self-adjoint. If the spectral measure $\mu_{x}$ of $x$ is atomic, then $\mu_{x}(\{\lambda\})$ is algebraic for every $\lambda\in \bC$.
\end{cor}

\begin{proof}
The case where $\lambda$ is transcendental follows from the above Theorem. So suppose that $\lambda$ is algebraic.
Define a polynomial
\[F_{\lambda}(t)=\prod_{\beta\in \spec(x),\beta\ne \lambda}(t-\beta).\]
Note that $F_{\lambda}$ has algebraic coefficients.
Then
\[1_{\lambda}(x)=\prod_{\beta\in \spec(x),\beta\ne \lambda}(\lambda-\beta)^{-1}F_{\lambda}(x).\]
So
\[\mu_{x}(\{\lambda\})=\prod_{\beta\in \spec(x),\beta\ne \lambda}(\lambda-\beta)^{-1}\tau(F_{\lambda}(x)).\]
We have that $\tau(F_{\lambda}(x))\in \Qbar$ by the preceding theorem, since $x^{k}$ is algebraically sofic for all $k\in \bN\cup\{0\}$.

\end{proof}

\begin{cor}\label{cor:non_alg_sof_direct_sum}
Suppose that $k_{1},k_{2}\in \bN$ and that $\gcd(k_{1},k_{2})=1$. Let $\alpha\in \bC$ be transcendental. Let $A=M_{k_{1}}(\bC)\oplus M_{k_{2}}(\bC)$ equipped with a trace 
\[\tau(x_{1},x_{2})=\alpha \tr(x_{1})+(1-\alpha)\tr(x_{2}).\]
If $x\in A$ is algebraically sofic with respect to $\tau,$ then $x\in \bC 1$. In particular, $A$ is not algebraically sofic. 
\end{cor}

\begin{proof}
Since $x+x^{*}$ and $i(x-x^{*})$ are algebraically sofic if $x$ is, we may assume that $x$ is self-adjoint.
Consider the spectral measure $\mu_{x}$ of $x$. Write $x=(x_{1},x_{2})$. By Theorem \ref{thm: store brand Andreas Thom} and Corollary \ref{cor:algebriac weights at algebraics}, $\mu_{x}$ is an atomic measure concentrated on algebraic integers and $\mu_{x}(\{\lambda\})$ is algebraic for every $\lambda\in \bC$. Let $\lambda\in \bC$, and let 
\[t_{i}=\frac{\dim(\ker(x_{i}-\lambda))}{k_{i}}\in \bQ.\]
Then
\[\mu_{x}(\{\lambda\})=\alpha(t_{1}-t_{2})+t_{2}.\]
Note that $\alpha$ is transcendental, whereas $t_{1},t_{2},\mu_{x}(\{\lambda\})$ are algebraic. Since algebraic numbers form a field, this forces $t_{1}=t_{2}$. Our assumptions on $k_{1},k_{2}$ thus forces that either 
\[\frac{\dim(\ker(x_{1}-\lambda))}{k_{1}}=\frac{\dim(\ker(x_{2}-\lambda))}{k_{2}}=0\]
or
\[\frac{\dim(\ker(x_{1}-\lambda))}{k_{1}}=\frac{\dim(\ker(x_{2}-\lambda)}{k_{2}}=1.\]
Since this holds for all $\lambda$ and $\mu_{x}$ is a probability measure, this forces $\mu_{x}$ to be a Dirac mass. Thus $x\in \bC 1$.    
\end{proof}

\section{Algebraic soficity preserved by graph products}
\label{sec: gp of as}

In this section, we show that the graph product of algebraically sofic tracial $*$-algebras is algebraically sofic.  In order to obtain the Galois bounded microstates for the graph product from Galois bounded microstates for the individual algebras, we use a construction based on conjugation by random permutation matrices from \cite{AIMSuperTeamI} (stated as Theorem \ref{thm: permutation model} below); this is the analog of Charlesworth and Collins' construction in the unitary case \cite{CC2021}, and the proof uses a similar technique as in the free case studied by \cite{ACDGM2021}.

To model graph products, we will need to force certain matrices to commute with each other, and certain matrices to be asymptotically free.  As in \cite{CC2021}, we will accomplish this by taking the models in a tensor product of several copies of $M_N(\bC)$, with matrices having only scalar components in certain tensor factors; in this way we can ensure that matrices which are meant to commute do so.  Heuristically, the index set of this tensor product will be a finite set of \emph{strings}.  Given a subalgebra of this larger product formed by replacing some of the tensor factors with copies of $\bC I_N$, we will think of its elements as corresponding to collections of beads on the strings where the algebra has a non-trivial factor.  Two algebras commute, then, if the beads representing their elements can slide past each other on this collection of strings.
For more detail on this picture, refer to \cite[\S3.2]{CC2021} or more generally \cite{MR2651902}.

The information of which tensor factors of a matrix are allowed to be non-scalar is determined by the vertex it corresponds to.
We will choose our set of strings and the assignments of vertices to sets of strings in such a way that matrices will share a string in common precisely when the graph product structure insists that the algebras they are modelling should be freely independent.
Given a prescribed finite graph $\cG$ it is always possible to choose a set $\cS$ and a relation $\sstrc$ with this with this property; one approach was given in \cite[Section 3.1]{CC2021}.

The matrices produced by our construction will all live in $M_N(\bC)^{\otimes \cS}$.  The inputs to the construction are deterministic matrices $X_j^{(N)}$ which are each assigned a certain vertex $\chi(j)$, such that $X_j^{(N)}$ is $M_N(\bC)^{\otimes \cS_v}$, viewed as a subalgebra of $M_N(\bC)^{\otimes \cS}$ in the standard way.  Each matrix $X_j^{(N)}$ with $\chi(j) = v$ will be conjugated by a random permutation matrix $\Sigma_v^{(N)}$ in $\bigotimes_{\cS_v} M_N(\bC)$ to produce a new random matrix $\uu{X}_j^{(N)}$ in $M_N(\bC)^{\otimes \cS}$.  When we apply this construction in the proof of Theorem \ref{thm: algebraic soficity intro}, $X_j^{(N)}$ will be a matrix approximation for some element of $A_{\chi(j)}$, more specifically some polynomial evaluated on microstates for our chosen generators of $A_{\chi(j)}$.

Theorem \ref{thm: permutation model} is a statement about certain polynomials in $\uu{X}_j^{(N)}$ given by $\cG$-reduced words with respect to the graph $\cG$ (see Definition \ref{def: G reduced word}).  The following theorem is a special case and slight reformulation of the main theorem of \cite{AIMSuperTeamI}.

\begin{thm} \label{thm: permutation model}
Let $\cG = (\cV, \cE)$ be a simple graph with vertex set $\cV$, $\cS$ be a finite set, and $\sstrc$ be as above so that $\cS_v \cap \cS_{v'} = \varnothing$ if and only if
$(v, v') \in \cE$.

For $N \in \bN$, let $\Delta_{N^{\#\cS}}$ be the conditional expectation onto the diagonal $*$-subalgebra $D_N$ of $\bigotimes_{\cS}M_N(\bC)$.

Let $N_k$ be a sequence of natural numbers with $N_k \to \infty$.  Let $\chi : [m] \to \cV$ be such that $\chi(1)\cdots\chi(m)$ is a $\cG$-reduced word, and for $i = 1, \ldots, m$ and $k \in \bN$, let $X_i^{(k)} \in \bigotimes_{S_{\chi(i)}} M_{N_k}(\bC)$ be a deterministic matrix, with $\sup_{k,i,j} \norm{X_i^{(k)}} < \infty$.

Further, let $\set{\Sigma_{v}^{(N)} : v \in \cV}$ be a family of independent uniformly random permutation matrices, with $\Sigma_v \in \bigotimes_{\cS_v} M_N(\bC)$, and write
\[
\uu{X}_i^{(k)} = \paren{\Sigma_{\chi(i)}^{(k)}}^* X_i^{(k)} \Sigma_{\chi(i)}^{(k)} \otimes I_{N_k}^{\otimes\cS\setminus\cS_{\chi(i)}} \in \bigotimes_{\cS} M_N(\bC).
\]
Then
\begin{equation}\label{eqn: graph ind ovre diagonal intro}
  \lim_{k \to \infty} \norm{\Delta_{N_k^{\#\cS}}[(\uu{X}_1^{(k)} - \Delta_{N^{\#\cS}}[\uu{X}_1^{(k)}]) \dots (\uu{X}_k^{(k)} - \Delta_{N^{\#\cS}}[\uu{X}_m^{(k)}])]}_2 = 0 \text{ almost surely.}
  \end{equation}
\end{thm}

Note in \cite{AIMSuperTeamI}, $\mathcal{C}$ rather than $\mathcal{V}$ is used for the set of vertices of $\mathcal{G}$.  In the notation of \cite{AIMSuperTeamI}, we have taken the diagonal matrices $\Lambda_{i,j}^{(n)}$ to be identity.  Moreover, rather than having matrices $X_{i,j}^{(N)}$ with $j = 1$, \dots, $\ell(i)$, we have a single matrix $X_i^{(N)}$ (we take $\ell(i) = 1$).  We used $m$ here rather than $k$ to denote the length of the word.  Finally, we rather than having a sequence $\uu{X}^{(N)}$ of matrices of size $N^{\# \cS}$, we consider a sequence $\uu{X}^{(N_k)}$ of matrices of size $N_k^{\# \cS}$; the theorem clearly still holds in this setting, since the proof is based on computing expectations and analyzing their dependence on $N$, for which can simply substitute $N_k$.

We are now ready to prove Theorem~\ref{thm: algebraic soficity intro}, which we restate here.

\begin{thm}[{Theorem ~\ref{thm: algebraic soficity intro}}]
\label{thm: gp of as is as}
Let $\cG = (\cV, \cE)$ be a finite simple graph and let $(A_v,\tau_v)$ for a $v \in \cV$ be a family of tracial $*$-algebras.  If each $(A_v,\tau_v)$ is algebraically sofic, then so is $\gp_{v \in \cG} (A_v,\tau_v)$.
\end{thm}

\begin{proof}
Suppose that $(A_v,\tau_v)$ for $v \in \cV$ are algebraically sofic, and let us prove that the graph product $\gp_{v \in \cG} (A_v,\tau_v)$ is algebraically sofic.

For each $v$, fix a generating tuple $y_v$ for $A_v$.  Fix Galois bounded sequences of microstates $\tilde{Y}_v^{(k)}$ in $M_{N_{v,k}}(\bC)$ for $y_v$.  Let $N_k = \prod_{v \in \cV} N_{v,k}$, and let $Y_v^{(k)} = \tilde{Y}_v^{(k)} \otimes I_{N_k / N_{k,v}}$.  Note that $Y_v^{(k)}$ is a Galois bounded sequence of microstates for $y_v$, but these microstates now come from the same matrix algebra $M_{N_k}(\bC)$ for all vertices $v$.

Let $\cS$ be a finite set and $\sstrc$ a relation between $\cS$ and $\cV$ so that for $v_1, v_2 \in \cV$, $(v_1, v_2) \in \cE$ if and only if $\cS_{v_1} \cap \cS_{v_2} = \emptyset$.
For each $v \in \cV$, fix some $s_v \in \cS_v$.
Let $\set{\Sigma_v^{(k)} \colon v \in \cV}$ be a family of independent uniformly random permutation matrices, with $\Sigma_v \in \bigotimes_{\cS_v}M_{N_k}(\bC)$.
Let
\[
  Z_v^{(k)} = \left[ (\Sigma_v^{(k)})^t (X_v^{(k)} \otimes I_N^{\otimes \cS_v \setminus \{s_v\}} ) \Sigma_v^{(k)} \right] \otimes I_N^{\otimes \cS \setminus \cS_v}.
\]

Let $y$ and $Z^{(k)}$ be the tuples obtained by concatenating the tuples $y_v$ and $Z_v^{(k)}$  respectively, over all $v \in \cV$.
It is immediate that each random outcome of $(Z^{(k)})_{k \in \bN}$ is Galois bounded.

It remains to show that almost surely $Z^{(k)}$ is a microstate sequence for $y$ and has asymptotically constant diagonal.  Being a microstate sequence means that for every non-commutative polynomial $p$, we have
\[
\lim_{k \to \infty} |\tr_{N_k^{\# \cS}} (p(Z^{(k)})) - \tau(p(y))| = 0,
\]
while being asymptotically constant on the diagonal means that
\[
\lim_{k \to \infty} \norm{\Delta_{N_k^{\#\cS}}[p(Z^{(k)}] - \tr_{N_k^{\# \cS}} (p(Z^{(k)})) I_{N_k^{\#\cS}}}_2 = 0.
\]
In fact, the combination of these two conditions is equivalent to
\begin{equation} \label{eq: combo independence and diagonal}
\lim_{k \to \infty} \norm{\Delta_{N_k^{\#\cS}}[p(Z^{(k)}] - \tau(p(y)) I_{N_k^{\# \cS}} }_2 = 0;
\end{equation}
this follows from the triangle inequality and the fact that $\tr_{N_k^{\# \cS}} \circ \Delta_{N_k^{\#\cS}} = \tr_{N_k^{\# \cS}}$.  By linearity, it suffices to check \eqref{eq: combo independence and diagonal} for a spanning set of polynomials.  Recall \cite[Remark 2.7]{CaFi17} that polynomials in $y$ are spanned by $1$ and polynomials of the form
\begin{equation} \label{eq: centered polynomial}
p(z) = (p_1(z_{\chi(1)}) - \tau(p_1(y_{\chi(1)}))) \dots (p_\ell(z_{\chi(\ell)}) - \tau(p_\ell(y_{\chi(\ell)})))
\end{equation}
for $\cG$-reduced words $\chi(1) \dots \chi(\ell)$, with $\ell \geq 1$.  The claim \eqref{eq: combo independence and diagonal} is immediate when $p = 1$.  Thus, it remains to show \eqref{eq: combo independence and diagonal} in the case when $p$ has the form \eqref{eq: centered polynomial}, and note that in this case the term $\tau(p(y))$ in \eqref{eq: combo independence and diagonal} vanishes by graph independence of $(y_v)_{v \in \cV}$.  Hence, our goal \eqref{eq: combo independence and diagonal} reduces to showing that almost surely
\begin{equation} \label{eq: limit thing 1}
\lim_{k \to \infty} \norm{\Delta_{N_k^{\#\cS}}[(p_1(Z_{\chi(1)}^{(k)}) - \tau(p_1(y_{\chi(1)}))) \dots (p_\ell(Z_{\chi(\ell)}^{(k)}) - \tau(p_\ell(y_{\chi(\ell)})))]}_2 = 0.
\end{equation}
Now we assumed that $Y_v^{(k)}$ is a microstate sequence for $y$ that is asymptotically constant on the diagonal, and $Z_v^{(k)}$ is obtained from $Y_v^{(k)}$ by tensoring with the identity and conjugating by a permutation matrix, and so
\[
\lim_{k \to \infty} \norm{\Delta_{N_k^{\# \cS}}[p_j(Z_{\chi(j)}^{(k)})] - \tau(p_j(y)) I_{N_k^{\# \cS}}}_2 = \lim_{k \to \infty} \norm{\Delta_{N_k}[p_j(Y_{\chi(j)}^{(k)})] - \tau(p_j(y)) I_{N_k}}_2 = 0.
\]
Thus, by swapping out each $\tau(p_j(y))$ term \eqref{eq: limit thing 1} for $\Delta_{N_k^{\# \cS}}[p_j(Z_{\chi(j)}^{(k)})]$, using the fact that $p_j(Z_{\chi(j)}^{(k)})$ is uniformly bounded in operator norm as $k \to \infty$, we obtain
\begin{multline*}
\lim_{k \to \infty} \Bigl \lVert \Delta_{N_k^{\#\cS}}[(p_1(Z_{\chi(1)}^{(k)}) - \tau(p_1(y_{\chi(1)}))) \dots (p_\ell(Z_{\chi(\ell)}^{(k)}) - \tau(p_\ell(y_{\chi(\ell)})))] \\
- \Delta_{N_k^{\#\cS}}[(p_1(Z_{\chi(1)}^{(k)}) - \Delta_{N_k^{\# \cS}}[p_1(Z_{\chi(1)}^{(k)})]) \dots (p_\ell(Z_{\chi(\ell)}^{(k)}) - \Delta_{N_k^{\#\cS}}(p_\ell(Z_{\chi(\ell)}^{(k)})))] \Bigr \rVert = 0,
\end{multline*}
so now the claim \eqref{eq: limit thing 1} to be proved reduces to
\begin{equation} \label{eq: limit thing 2}
\lim_{k \to \infty} \norm{\Delta_{N_k^{\#\cS}}[(p_1(Z_{\chi(1)}^{(k)}) - \Delta_{N_k^{\# \cS}}[p_1(y_{\chi(1)})]) \dots (p_\ell(Z_{\chi(\ell)}^{(k)}) - \Delta_{N_k^{\#\cS}}(p_\ell(Z_{\chi(\ell)}^{(k)})))] }_2 = 0.
\end{equation}
Now we can apply Theorem~\ref{thm: permutation model}, taking
\[
X_j^{(k)} = p_j(Y_{\chi(j)}^{(k)}) \otimes I_N^{\otimes \cS_v \setminus \{s_v\}},
\]
so that
\[
\uu{X}_j^{(k)} = (\Sigma_{\chi(j)}^{(k)})^* (p_j(Y_{\chi(j)}) \otimes I_N^{\otimes \cS_v \setminus \{s_v\}}) \Sigma_{\chi(j)}^{(k)} =  p_j((\Sigma_{\chi(j)}^{(k)})^*(Y_{\chi(j)} \otimes I_N^{\otimes \cS_v \setminus \{s_v\}})\Sigma_{\chi(j)}^{(k)}) = p_j(Z_{\chi(j)}^{(k)}).
\]
Thus, Theorem \ref{thm: permutation model} implies that \eqref{eq: limit thing 2} holds, which completes the proof.
\end{proof}

\section{Strong $1$-boundedness for graph products} \label{sec: proof of main theorem}

Strong 1-boundedness is a von Neumann algebraic property introduced by Jung in \cite{Jung2007}.
It implies the lack of a robust space of microstates up to conjugacy for any generating set of a von Neumann algebra.
This typically is achieved when the von Neumann algebra is hyperfinite (see in connection, \cite{JungRegularity, JungHyperFiniteIneq}) or admits algebraic rigidity in the form of abundant commutation (see \cite{GePrime, geshen, VoiculescuPropT}) or existence of diffuse regular hyperfinite subalgebras (see \cite{Voiculescu1995, Hayes2018}), or even in the analytic setting of Property (T) which allows for discretizing the microstate space (see \cite{JungS, propts1b}).
On the other hand, strong 1-boundedness implies that every generating set has microstates free entropy dimension $\delta_0(x)=1$, hence the free group factors are not strongly 1-bounded. 
Hayes refined this notion by extracting a numerical invariant, implicit in \cite{Jung2007}, for von Neumann algebras called the 1-bounded entropy $h$ (see \cite{Hayes2018}).
This is the main framework in which the modern theory of strong 1-boundedness is carried out.
Non-strongly 1-bounded algebras often exhibit indecomposability relative to strongly 1-bounded subalgebras, which can be used to prove non-isomorphism results or rule out possible structural properties.
As a precise example, non-strongly 1-bounded algebras cannot be generated by two strongly 1-bounded subalgebras with diffuse intersection.
Another application is a free absorption theorem for strongly 1-bounded subalgebras in free products (\cite{FreePinsker}).

Such indecomposability phenomena in the setting of groups in many instances can be encapsulated in $L^2$-invariants, such as the first $L^2$-Betti number (see \cite{LuckBook}).
This cohomological invariant has been of extreme use in the analytic study of groups, and has been increasingly incorporated as far as possible into the study of von Neumann algebras due to its rich applications (see \cite{ConnesShl, PetersonDeriva}).
Having positive first $L^2$-Betti number automatically implies the lack of the sort of algebraic rigidity described above in the group level.
See \cite{PetersonThom} for such results.
The relationship between the first $\ell^2$-Betti number and free entropy theory is a difficult subject that has been heavily investigated (\cite{ConnesShl, JungL2B, Shl2015, vanishingl2bettis1b}).
Strong 1-boundedness for Connes-embeddable group von Neumann algebras is believed to coincide with vanishing first $\ell^2$-Betti number for the group.
However, this has been checked only in certain cases, particularly in one direction as outlined in \cite{DimaLowerEstimates, Shl2015}, and remains a challenging open problem.

Given a tracial von Neumann algebra $(M,\tau)$ and $N$ a von Neumann subalgebra of $M$, the \textbf{$1$-bounded entropy of $N$ in the presence of $M$ is denoted $h(N:M)$.} We set $h(M)=h(M:M)$ and call this the \textbf{$1$-bounded entropy of $M$}. Roughly speaking, the quantity $h(N:M)$ is a measurement of ``how many" finite-dimensional approximations of $N$ there are which extend to $M$,
We will not need the technical definition of $1$-bounded entropy, and refer the reader to \cite[Definition 2.2 and Definition A.2]{Hayes2018} for the precise definition.
We enumerate below the most essential properties of this quantity for our purposes:

\begin{enumerate}
    \item  (see  \cite[\S 2.3.3]{propts1b})\label{fact 1}
 $h(N_{1}:M_{1})\leq h(N_{2}:M_{2})$ if $N_{1}\subset N_{2}\subset M_{2}\subset M_{1}$.
 \item  \label{joins}(see \cite[Lemma A.12]{Hayes2018})
  $h(N_1\vee N_2:M)\leq h(N_1:M)+h(N_2:M)$ if $N_1,N_2\subset M$ and $N_1\cap N_2$ is diffuse.
  In particular, $h(N_1\vee N_2)\leq h(N_1)+h(N_2)$.
 \item \label{normalize 1bEnt} (see \cite{Hayes2018}) $h(N_1:N_2)\leq h(W^*(\mathcal{N}_{N_2}(N_1)):N_2)$ if $N_1\subset N_2$ is diffuse.
 \item If $N\subseteq M$ and $N$ is hyperfinite, then $h(N:M)\leq 0$. \label{hyperfinite}
\end{enumerate}
We will also need Voiculescu's microstates free entropy dimension =$\delta_{0}(x)$ of a self-adjoint tuple $x$ in a tracial von Neumann algebra, define by Voiculescu \cite{Voiculescu1996}. We will need to allow $x$ to be an infinite tuple, as opposed to a finite tuple in Voiculescu's original definition. It is well known to experts how to extend the definition to this setting, for a precise discussion see e.g. the discussion in Section 4 of \cite{propts1b}. We use $\underline{\delta}_{0}(x)$ for the version of microstates free entropy dimension where we replace a limit supremum in the definition with a limit infimum.

In contrast to the rest of the paper, we will need to restrict ourselves to self-adjoint generating tuples. For an integer $r\in \bN$, we let $\bC\ang{S_{1},\cdots,S_{r}}$ be the algebra of noncommutative polynomials in abstract variables $S_{1},\cdots,S_{r}$. We give $\bC\ang{S_{1},\cdots,S_{r}}$ the unique $*$-structure which makes each $X_{j}$ self-adjoint. Given a von Neumann algebra $M$ and a  $x\in M_{s.a.}^{r}$ there is a unique $*$-homomorphism
\[\ev_{x}\colon \bC\ang{S_{1},\cdots,S_{r}}\to M\]
satisfying $\ev_{x}(S_{j})=x_{j}$. We set $P(x)=\ev_{x}(P)$ for $P\in \bC\ang{S_{1},\cdots,S_{r}}$. We will use $\cO\ang{S_{1},\cdots,S_{r}}$, $\Qbar\ang{S_{1},\cdots,S_{r}}$ etc. for the noncommutative polynomials in $r$-variables whose coefficients are in $\cO,\Qbar$ etc.

\subsection{Proof of Theorem \ref{thm: main theorem intro0} } \label{subsec: proof of main theorem}

We are now ready to prove Theorem~\ref{thm: main theorem intro0}.
For simplicity, we treat its parts \ref{item: connected graphs intro} and \ref{item: disconnected graphs intro} separately from part \ref{item: finite-dimensional case}.

\begin{thm}
\label{thm: gp of diffuse}
Let $\cG=(\cV,\cE)$ be a graph with $\#\cV > 1$, and for each $v\in \cV$, let $(M_{v},\tau_{v})$ be a tracial $*$-algebra. Let $(M,\tau)=\gp_{v\in \cG}(M_{v},\tau_{v})$.
\begin{enumerate}[(1)]
    \item If each $M_{v}$ is diffuse and $\cG$ is connected, then $M$ is strongly $1$-bounded (in fact has $1$-bounded entropy at most zero). \label{item: connected and diffuse destroys entropy}
    \item If each $M_{v}$ is diffuse and Connes embeddable, and $\cG$ is disconnected, then there is an index set $I$ and a generating tuple $x\in M_{s.a.}^{I}$ so that $\delta_{0}(x)>1$. In particular, $M$ is not strongly $1$-bounded. \label{item:disconnected and diffuse case}
\end{enumerate}
\end{thm}

\begin{proof}

(\ref{item: connected and diffuse destroys entropy}):
Since $\cG$ is connected, we can find a walk $v_{1},v_{2},\cdots,v_{k}$ which visits every vertex of $\cG$ at least once.
Let us denote by $M_{\leq j}$ the algebra generated by $M_{v_1}, \ldots, M_{v_j}$ within $M$.
We claim that $h(M_{\leq j})\leq 0$ for all $j\geq 2$.

Because $M_{v_1}$ and $M_{v_2}$ are diffuse, we may choose diffuse abelian subalgebras $A_1 \leq M_{v_1}$ and $A_2 \leq M_{v_2}$.
Using Properties (\ref{fact 1}), (\ref{normalize 1bEnt}), (\ref{hyperfinite}) of $1$-bounded entropy,
\begin{align*}
h(M_{v_1}\vee A_{2}) & =h(M_{v_1}\vee A_{2}:M_{v_1}\vee A_{2}) \\
&\leq h(W^{*}(\cN_{M_{v_1}\vee A_{2}}(A_{2})):M_{v_1}\vee A_{2})) \\
&\leq h(A_{2}:M_{v_1}\vee A_{2}) \\
&\leq 0.
\end{align*}
Similarly, $h(A_{1}\vee M_{v_2})\leq 0$. As
\[
M_{v_1}\vee A_{2}\cap (A_{1}\vee M_{v_2})\supseteq A_{1}\vee A_{2},
\]
we know that $M_{v_1}\vee A_{2}\cap (A_{1}\vee M_{v_2})$ is diffuse. Thus, by Property (\ref{joins}) of $1$-bounded entropy:
\[h(M_{\leq 2})\leq h(M_{v_1}\vee A_{2})+h(A_{1}\vee M_{v_2})\leq 0.\]

For the general case, note for every $2\leq i < n$, we have $(M_{v_i}\vee M_{v_{i+1}})\cap M_{\leq i}\supseteq M_{v_i}$, which is diffuse. Thus by Property (\ref{joins}) of $1$-bounded entropy:
\[h(M_{\leq i+1})\leq h(M_{v_i}\vee M_{v_{i+1}})+h(M_{\leq i})\leq h(M_{\leq i}),\]
the last inequality following from an argument identical to the case of $M_{\leq 2}$.
We thus inductively see that $h(M)=h(M_{\leq n})\leq 0$.

(\ref{item:disconnected and diffuse case}):
Let $\cV_1, \ldots, \cV_l$ be the connected components of $\cG$, and note that $l \geq 2$ by assumption.
Let $M_{i}$ be the graph product corresponding to the subgraph induced by $\cV_i$. 
Then the $(M_{i})_{i=1}^{l}$ are freely independent.
Let $x_{i}\in (M_{i})_{s.a.}^{J_{i}}$ be a generating tuple. Set $J=\bigsqcup_{i}J_{i}$, and let $x\in M_{s.a.}^{J}$ be defined by $x\big|_{J_{i}}=x_{i}$.
Since each $M_{v}$ is embeddable, we know that $M_{i}$ is embeddable by \cite{casperscep}.
Since $M_{i}$ is diffuse, this implies by the proof of
 \cite[Corollary 4.7]{JungRegularity}) that $\underline{\delta}_{0}(x_{i})\geq 1$.
Thus, by the proof of \cite{Voiculescu1998}[Remark 4.8],
\[\delta_{0}(x)=\delta_{0}(x_{1})+\sum_{i=1}^{l}\underline{\delta}_{0}(x)\geq l>1,\]
the last inequality following as $\cV$ is disconnected.

\end{proof}

To deduce strong $1$-boundedness from vanishing first $L^{2}$-Betti number, we will apply  the results in \cite{Shl2015} which require positive of certain Fuglede--Kadison pseudo-determinants associated to our relations.
To get this positivity, we will use Theorem \ref{thm:FKD positive intro} which requires polynomials with algebraic coefficients.
 This will force us to reduce general relations among generators for our tracial von Neumann algebras to only relations that have algebraic coefficients. For this, the following lemma will be useful.

\begin{lem}\label{lem: finite presentation lemma}
Let $(M,\tau)$ be a tracial von Neumann algebra and $x=(x_{1},\cdots,x_{r})\in M_{s.a.}^{r}$.
Suppose that for all
$P\in \Qbar{S_{1},\cdots,S_{r}}$
we have that $\tau(P(x))\in \Qbar$.
Let $\ev_{x}\colon \bC\ang{S_{1},\cdots,S_{r}}\to M$ be the $*$-homomorphism $\ev_{x}(P)=P(x).$ Then:
\begin{enumerate}[(a)]
\item $\ker(\ev_{x})$ is the complex linear span of $\ker(\ev_{x})\cap \overline{\bQ}\ang{S_{1},\cdot,S_{r}}$. \label{item:relations reduce to algebraic ones}
\item If $\ker(\ev_{x})$ is finitely generated as a two-sided ideal, then there is a finite set \[F\subseteq \overline{\bQ}\ang{S_{1},\cdots,S_{r}}\] which generates $\ker(\ev_{x})$ as a two-sided ideal.
\label{item: finite algebraic presentation}
\end{enumerate}

\end{lem}

\begin{proof}

(\ref{item:relations reduce to algebraic ones}).
Let $P\in \ker(\ev_{x})$. Then there are monic monomials $m_{1},\cdots,m_{d}$ and $\lambda_{1},\cdots,\lambda_{d}\in \bC$ with $P=\sum_{j=1}^{d}\lambda_{j}m_{j}$. Let $A\in M_{d}(\bC)$ be the matrix whose $ij^{th}$ entry is $\tau(m_{j}(x)^{*}m_{i}(x))$.  Since $\tau$ is a state, $A$ is positive semidefinite.  Let
\[\lambda=\begin{bmatrix}
    \lambda_{1}\\
    \lambda_{2}\\
    \vdots\\
    \lambda_{d}
\end{bmatrix}\in \bC^{d}.\]
By direct calculation,
\[\|P(x)\|_{2}^{2}=\ip{A\lambda,\lambda}.\]
Since $P(x)=0$ and $A$ is positive semidefinite, we know that $A\lambda=0$. Observe that $A$ has algebraic entries, by assumption. Since $A$ has algebraic entries, it follows from linear algebra that the kernel of $A$ (regarded as a linear transformation on $\bC^{d}$) has a basis $v_{1},\cdots,v_{s}\in \overline{\bQ}^{d}$. Choose complex numbers $\alpha_{1},\cdots,\alpha_{s}$ so that
\[\lambda=\sum_{k=1}^{s}\alpha_{k}v_{k}\]
For $k=1,\cdots,s$ write $v_{k}=(v_{kj})_{j=1}^{d}\in \overline{\bQ}^{d}$ and set $P_{k}=\sum_{j=1}^{d}v_{kj}m_{j}$.
Since $v_{k}\in \ker(A)$, we have that
\[\|P_{k}(x)\|_{2}^{2}=\ip{Av_{k},v_{k}}=0.\]
So $P_{k}\in \ker(\ev_{x})\cap \overline{\bQ}\ang{S_{1},\cdots,S_{r}}$ and
\[P=\sum_{k=1}^{s}\alpha_{k}P_{k}.\]

(\ref{item: finite algebraic presentation})
Suppose that $F_{1},\cdots,F_{p}\in \bC\ang{S_{1},\cdots,S_{r}}$
generate $\ker(\ev_{x})$ as a two-sided ideal. By (\ref{item:relations reduce to algebraic ones}), we may find a $t\in \bN$
and $\lambda_{ij}\in \bC$,
$F_{ij}\in \ker(\ev_{x})\cap \overline{\bQ}\ang{S_{1},\cdots,S_{r}}$
for $1\leq i\leq p,1\leq t\leq j$ so that
\[F_{i}=\sum_{j}\lambda_{ij}F_{ij}.\]
Then $\{F_{ij}\}_{1\leq i\leq p,1\leq t\leq j}$ generate
$\ker(\ev_{x})$ as a two-sided ideal.
\end{proof}

We will also need to pass to direct sums of algebras for which the above lemma applies.  For this we use the following lemma.

\begin{lem}\label{lem:annoying direct sum lemma}
Let $A_{1}$, $A_2$ be two $*$-algebras which are generated by $x\in (A_1)_{s.a.}^{r_1}$ and $y\in (A_2)_{s.a.}^{r_2}$.
Suppose that
    \[
        E_{1}\subseteq \bC\ang{S_{1},\cdots,S_{r_{1}}} \qquad \text{ and } \qquad E_2 \subseteq \bC\ang{T_{1},\cdots,T_{r_{2}}}
    \]
generate $\ker(\ev_{x})$ and $\ker(\ev_{y})$, respectively, as two-sided ideals. Denote
    \begin{align*}
        \tilde{x}&:=(1\oplus 0, x_1\oplus 0,\ldots, x_{r_1}\oplus 0)\\
        \tilde{y}&:=(0\oplus 1, 0\oplus y_1,\ldots, 0\oplus y_{r_2}).
    \end{align*}
Then $\ker(\ev_{\tilde{x},\tilde{y}})\subset \bC\ang{S_0,S_1,\ldots, S_{r_1}, T_0, T_1,\ldots, T_{r_2}}$ is generated as a two-sided ideal by the union
    \begin{align*}
        &\{S_{0}P:P\in E_{1}\}\\
        \cup & \{T_{0}P:P\in E_{2}\}\\
        \cup & \{S_{i}T_{j}\colon 0\leq i\leq r_{1},\ 0\leq j\leq r_2\}\\
        \cup & \{S_0S_i - S_i, S_i S_0 - S_i\colon 1\leq i\leq r_1\}\\
        \cup & \{T_0 T_j - T_j, T_j T_0 - T_j \colon 1\leq j\leq r_2\}\\
        \cup & \{S_0 + T_0 -1\}.
    \end{align*}
\end{lem}
\begin{proof}
Let $J$ be the two-sided ideal in $\bC\ang{S_0,\ldots, S_{r_1}, T_0, \ldots, T_{r_2}}$ generated by the above union, and let $B:=\bC\ang{S_{0},\cdots,S_{r_1},T_0,\ldots, T_{r_2}}/J$. Then $J\subseteq \ker(\ev_{\tilde{x},\tilde{y}})$ and this inclusion induces a unique homomorphism
    \[
        \psi\colon B\to A_{1}\oplus A_{2}
    \]
satisfying $\psi\circ q=\ev_{\tilde{x},\tilde{y}}$, where $q$ is the quotient map onto $B$. To prove the lemma, it suffices to show that this homomorphism is an isomorphism.

To see this, let $z_{1}=S_{0}+J,z_{2}=T_{0}+J$. Then $z_{1},z_{2}$ are orthogonal projections which sum to $1$. Observe that for all $P\in \bC\ang{S_{1},\cdots,S_{r_{1}}}$ we have
\[q(S_{0}P(S_{1},\cdots,S_{r_{1}}))=q(P(S_{0}S_{1},\cdots,S_{0}S_{r_{1}}))\]
Thus for $P\in E_1$ we have
    \[
        q(P(S_{0}S_{1},\cdots,S_{0}S_{r_{1}}))=q(S_{0}P(S_{1},\cdots,S_{r_{1}}))=0.
    \]
Since $E_{1}$ generates $\ker(\ev_{x})$ as a two-sided ideal and $A_{1}\cong \bC\ang{S_{1},\cdots,S_{r_{1}}}/\ker(\ev_{x_{1}})$, we may find a unique homomorphism $\phi_{1}\colon A_{1}\to z_{1}B$ satisfying  $\phi_{1}(P(x))=S_{0}P+J$ for all $P\in \bC\ang{S_{1},\cdots,S_{r_{1}}}$. Similarly, we may find a unique homomorphism $\phi_{2}\colon A_{2}\to z_{2}B$ satisfying
$\phi_{2}(P(y))=T_0 P+J$ for all $P\in \bC\ang{T_{1},\cdots,T_{r_{2}}}$. The relations imposed on $B$ imply that $z_{i}$ acts as the identity on the image of $\phi_{i}$. Since $z_{1},z_{2}$ are orthogonal projections which sum to $1$, this implies that the map $\phi\colon A_{1}\oplus A_{2}\to B$ defined by $\phi(a_{1},a_{2})=\phi_{1}(a_{1})+\phi_{2}(a_{2})$ is a homomorphism. Moreover, $\phi$ is the inverse to $\psi$. Thus $\psi$ is an isomorphism, as desired.
\end{proof}

We are now ready to prove a general theorem from which we will quickly deduce Theorem \ref{thm: main theorem intro0} (\ref{item: finite-dimensional case}) as a corollary. For this we need the first $L^{2}$-Betti number of a $*$-subalgebra of a tracial von Neumann algebra. The $L^{2}$-Betti number of von Neumann algebras was first defined in \cite[Definition 2.1]{ConnesShl} in terms of homology.  Thom later gave a definition in terms of cohomology, see \cite[Section 1]{ThomL2cohom}.

\begin{thm}[{Theorem~\ref{thm: main theorem intro0} (\ref{item: finite-dimensional case})}]
\label{thm: gp of as is s1b}
Let $\cG=(\cV,\cE)$ be a graph with $\#\cV > 1$, and for each $v\in \cV$, let $(M_{v},\tau_{v})$ be a tracial $*$-algebra. Let $(M,\tau)=\gp_{v\in \cG}(M_{v},\tau_{v})$.
Suppose that for all $v\in \cV$, we can write $M_{v}=\bigoplus_{i=1}^{g_{v}}M_{v,i}$ with $g_{v,i}\in \bN$. Further assume that:
\begin{itemize}
\item $\tau(1_{M_{v,i}})\in \bQ$, for all $v\in \cV$ and $i=1,\cdots,g_{v,i}$,
    \item for all $v$ and all $1\leq i\leq g_{v},$ there is a $x_{v,i}\in M_{v,i}^{r_{v,i}}$ which is algebraically sofic and generates $M_{v,i}$ as a von Neumann algebra,
    \item for all $v\in \cV,1\leq i\leq g_{v}$, $\tau_{v}(P(x_{v,i}))\in \overline{\bQ}$ for all $P\in \Qbar\ang{T_{1},T_{1}^{*},\cdots,T_{r},T_{r}^{*}}$
    \item $\ker(\ev_{x_{v,i}})$
     is finitely generated as a two-sided ideal in $\bC\ang{T_{1},T_{1}^{*},\cdots,T_{r_{v,i}},T_{r_{v,i}}^{*}}$ for all $v\in \cV,$ $1\leq i\leq g_{v}$.
    \end{itemize}
Let $A$ be the $*$-subalgebra of $M$ generated by $\bigcup_{v\in V}M_{v}$. If $\beta^{1}_{(2)}(A,\tau)=0$, then $M$ is strongly $1$-bounded.
\end{thm}

\begin{proof}
First, notice that by considering the real and imaginary parts of coordinates of $x_{v,i}$, we may assume that $x_{v,i}$ is a tuple of self-adjoint elements.
Let $A$ be the $*$-algebra generated  by the $M_{v}$ for $v\in \cV$. Then $A$ is finitely presented, namely there is an $r\in \bN$, a tuple $x\in A_{s.a.}^{r}$ so that:
\begin{enumerate}
    \item the evaluation homomorphism $\ev_{x}\colon \bC\ang{S_{1},\cdots,S_{r}}\to A$ given by $\ev_{x}(P)=P(x)$ is surjective,
    \item if $J=\ker(\ev_{x})$, then $J$ is finitely generated as a two-sided ideal say by $(F_{1},\cdots,F_{l})$.
\end{enumerate}
In fact,
we may choose $x$ to be algebraically sofic and to choose $F_{i}\in \Qbar\ang{S_{1},\cdots,S_{r}}$.
One way to see this is as follows.

By assumption, $\ker(\ev_{x_{v,i}})$ can be generated as a two-sided ideal by $(F^{(v,i)}_{h})_{h=1}^{k_{v,i}}$. Our assumptions on traces of algebraic polynomials in $x_{v,i}$ and
 Lemma \ref{lem: finite presentation lemma} (\ref{item: finite algebraic presentation}) implies  we can choose $F^{(v,i)}_{h}$ to be in $\Qbar\ang{S_{1},\cdots,S_{r_{v,i}}}$. Set $r_{v}=g_{v}+\sum_{i}r_{v,i}$ and let
\[x_{v}=(1\oplus 0^{\oplus g_{v}-1},0\oplus 1\oplus 0^{\oplus g_{v}-2},\cdots, 0^{\oplus g_{v}-1}\oplus 1,x_{v,1}\oplus 0^{\oplus g_{v}-1},0\oplus x_{v,2}\oplus 0^{\oplus g_{v}-2},\cdots,0^{\oplus g_{v}-1}\oplus x_{v,g_{v}}).\]

 Let $r=\sum r_{v}$, $k_{v}=\sum_{i}k_{v,i}$ and $x\in M_{s.a.}^{r}$ be given by concatenating the $x_{v}$. We relabel the abstract variables $S_{1},\cdots,S_{r}$ as $S^{(v)}_{j}$ with $v\in \cV$ and $1\leq j\leq r_{v}$, and set $S^{(v)}=(S^{(v)}_{j})_{1\leq j\leq r_{v}}$.
 By iterated applications of Lemma \ref{lem:annoying direct sum lemma}, we can find a  finite tuple $F^{(v)}\in (\Qbar\ip{S_{1},\cdots,S_{r_{v}}})^{\oplus k_{v}}$ which generates $\ker(\ev_{x_{v}})$ as a two-sided ideal in $\bC\ang{S_{1},\cdots,S_{r_{v}}}$.

To generated $\ker(\ev_{x})$, we need to take all the $F^{(v)}(S^{(v)})$ and also polynomials of the form
\[S^{(v_1)}_{j}S^{(v_2)}_{p}-S^{(v_2)}_{p}S^{(v_1)}_{j}\]
for all $(v_1,v_2)\in \cE$, $1\leq j\leq k_{v_1}$,$1\leq p\leq k_{v_2}$.
Then these polynomials generate $\ker(\ev_{x})$ and have algebraic coefficients. Moreover, $x$ is an algebraically sofic tuple by the proof of Theorem \ref{thm: finite dim algebras are alg sofic}, and Theorem \ref{thm: gp of as is as}.

Let $F=(F_{1},\cdots,F_{l})\in \Qbar\ang{S_{1},\cdots, S_{r}}^{\oplus l}$. For $i=1,\cdots,r$ define
\[(\partial_{i} F)\in M_{l,1}(\Qbar\ang{S_{1},\cdots, S_{r}}\otimes \Qbar\ang{S_{1},\cdots, S_{r}})\]
by
$(\partial _{i}F)_{j1}=\partial_{i}F_{j}.$
Here $\partial_i : \Qbar\ang{S_1, \ldots, S_r} \to \Qbar\ang{S_1, \ldots, S_r} \otimes \Qbar\ang{S_1, \ldots, S_r}$ are Voiculescu's free difference quotients (see \cite[Section 2]{entropy5}), i.e., the unique derivations with $\partial_i(T_j) = \delta_{i=j}1\otimes 1$.
Finally, set
\[D_{F}=\begin{bmatrix}
S_{1}\otimes 1-1\otimes S_{1} & S_{2}\otimes 1-1\otimes S_{2}&\cdots &S_{r}\otimes 1-1\otimes S_{r}\\
\partial_{1} F&\partial_{2}F&\cdots&\partial_{r}F
\end{bmatrix}\]
Which is an element of $ M_{l+1, r}(\Qbar\ang{T_{1},\cdots, T_{r}}\otimes \Qbar\ang{T_{1},\cdots, T_{r}})$.
It is then a folklore result (see e.g the proofs of \cite[Theorem 1.1]{vanishingl2bettis1b}, \cite[Lemma 4.1]{BrannanVergnioux}) that
\[\dim_{M\overline{\otimes}M^{op}}(\ker(D_{F}(x)))=\beta^{1}_{(2)}(A,\tau)=0.\]
Thus if $\beta^{1}_{(2)}(A,\tau)=0$, then $D_{F}(x)$ is injective, and so $\mu_{|D_{F}(x)|}(\{0\})=0$.
Since $x$ is algebraically sofic, it follows by Proposition \ref{prop: alg sofic tensor} and Theorem \ref{thm:FKD positive intro} that
\[\Det_{M}^{+}(D_{F}(x))>0.\]
Moreover, $F(x)=0$ by construction.
Hence it follows by \cite[Theorem 1.5]{Shl2015}, \cite[Theorem 1.2]{vanishingl2bettis1b} that $M$ is strongly $1$-bounded.
\end{proof}

We remark that this theorem implies Theorem \ref{thm: main theorem intro0} (\ref{item: finite-dimensional case}) by taking each $M_{v}$ to be finite-dimensional tracial algebras where every central projection has rational trace. Indeed in this case we may take each $M_{v,i}$ to be a matrix algebra. Now use the isomorphism $M_{n}(\bC)\cong L(\bZ/n\bZ)\rtimes \bZ/n\bZ$ from Proposition \ref{prop:construction of GBM for matrices}. The generators for $M_{n}(\bC)$ given in Proposition \ref{prop:construction of GBM for matrices} (\ref{item:GNS matrix entries}) are algebraically sofic, by Proposition \ref{prop:construction of GBM for matrices}, and it is direct to check that monic monomials in these generators have algebraic traces.

\bibliographystyle{amsalpha}
\bibliography{graphproducts}

\end{document}